\documentclass{article}
\usepackage{amssymb,amsmath,bm,geometry,graphics,graphicx,url,color}

\def\no{\noindent}
\def\pmatrix{\left(\begin{array}}
\def\endpmatrix{\end{array}\right)}
\newtheorem{theo}{Theorem}

\newtheorem{cor}{Corollary}
\newtheorem{rem}{Remark}
\newtheorem{defi}{Definition}

\newtheorem{prop}[defi]{Proposition}

\usepackage{color}

\title{Long-time analysis of  extended RKN
integrators for  Hamiltonian systems with a solution-dependent high
frequency}%\thanks{This work is supported by Universit\`a di Firenze (project ``Risoluzione numerica di problemi Hamiltoniani ed applicazioni'') and NSFC (Grant No.\,11571128).}}

\author{Bin Wang\,
\footnote{School of Mathematical Sciences, Qufu Normal University,
Qufu 273165, P.R. China; Mathematisches Institut, University of
T\"{u}bingen, Auf der Morgenstelle 10, 72076 T\"{u}bingen, Germany.
The research is supported in part by the Alexander von Humboldt
Foundation and by the Natural Science Foundation of Shandong
Province (Outstanding Youth Foundation) under Grant ZR2017JL003.
E-mail:~{\tt wang@na.uni-tuebingen.de} } \and Xinyuan
Wu\thanks{School of Mathematical Sciences, Qufu Normal University,
Qufu 273165, P.R. China; Department of Mathematics, Nanjing
University, Nanjing 210093, P.R. China. The research is supported in
part by the National Natural Science Foundation of China under Grant
11671200. E-mail:~{\tt xywu@nju.edu.cn}} }

\begin{document}
\maketitle

\begin{abstract} In this paper, we analyse the long-time behaviour of the extended
RKN (ERKN) integrators for solving highly oscillatory Hamiltonian
systems with a slowly varying, solution-dependent high frequency. We
prove that a symmetric ERKN integrator approximately conserves a
modified action and a modified total energy over long time intervals
based on the technique of varying-frequency modulated Fourier
expansion.  An illustrative numerical experiment is carried out and
the numerical results strongly support the theoretical analysis
presented in this paper. As a byproduct of this work, similar
long-time behaviour is also investigated for an RKN method.
\medskip

\no{\bf Keywords:}  Long-time analysis, Extended RKN integrators,
Hamiltonian systems, Solution-dependent high frequency, RKN methods

\medskip
\no{\bf MSC:}65P10, 65L05

\end{abstract}

%=============================================================================================
\section{Introduction}
It is well known that Hamiltonian systems frequently arise in many
areas of science and engineering  such as applied mathematics,
molecular biology, electronics, chemistry, astronomy, mechanics and
quantum physics.
 Numerical simulations of such
systems usually  require a long-time integration. This paper is
devoted to the analysis of the long-time behaviour for
  extended Runge--Kutta--Nystr\"{o}m   integrators (ERKN)   when   applied to the following highly oscillatory Hamiltonian
system
\begin{equation}
 \left\{\begin{aligned}
 &\dot{q}=\nabla_p H(q,p),\qquad \ \ q(0)=q^{0},\\
 &\dot{p}=-\nabla_q H(q,p),\qquad  p(0)=p^{0}, \end{aligned}\right.
\label{H-s}%
\end{equation}
with the Hamiltonian
\begin{equation}H(q,p)=\frac{1}{2}\big(|p_1|^2+ |p_2|^2+ \frac{\omega(q_1)^2}{\epsilon^2} |q_2|^2\big)+U(q),\label{H}%
\end{equation}
where $|\cdot|$ denotes the Euclidean norm,  $\epsilon$ is a small
parameter satisfying $0<\epsilon\ll1$, $\omega(q_1)$ stands for the
frequency function, and the vectors $p = (p_1, p_2) \in
\mathbb{R}^{d_1}\times \mathbb{R}^{d_2}$ and $q = (q_1, q_2)\in
\mathbb{R}^{d_1}\times \mathbb{R}^{d_2}$ represent the positions and
momenta, respectively. It is assumed that the frequency function
$\omega(q_1)$ and the potential $U(q)$ are smooth in the sense that
all their derivatives are bounded independently of $\epsilon$.
Following \cite{Hairer16},  the frequency $\omega(q_1)$ is   also
assumed to be slowly changing and has a fixed positive lower bound
independently of $q_1$, i.e., $\omega(q_1)\geq1$.

Similarly to \cite{Hairer16},  we choose
 the initial values as
 \begin{equation}
\begin{array}
[c]{ll}%
&q_1(0) = O(1),\qquad q_2(0) = O(\epsilon), \\
&p_1(0) = O(1),\qquad p_2(0) = O(1),
\end{array}
  \label{Initial val}%
\end{equation}
which imply that the total energy $H(q,p)$ is bounded by a constant
independent of $\epsilon$. As is known, the action   of the system
\eqref{H-s} is given by
\begin{equation*}
I(q,p)=\frac{1}{2}\frac{|p_2|^2}{\omega(q_1)}  +\frac{\omega(q_1)}{2\epsilon^2}|q_2|^2, %\label{action}%
\end{equation*}
which is $O(1)$ at the initial values \eqref{Initial val}. It has
been proved in \cite{Reich00} that this action is nearly conserved
in the case of a single fast degree of freedom $(d_2 = 1)$. With
regard to the almost conservation of the action in the general case
$d_2 \geq 1$,   it has been shown in \cite{Hairer16} that the
considered Hamiltonian system has the action as an adiabatic
invariant over long times (the details are referred to Theorem 3.2
in \cite{Hairer16}).

It is noted that  the highly oscillatory system \eqref{H} has a
slowly varying, solution-dependent high frequency
$\frac{\omega(q_1)}{\epsilon}$. This system was first studied
analytically in  \cite{Rubin57}. Then it has been researched in many
publications and we refer to \cite{Arnold97,Bornemann98,Cotter04}
for example.  It is clear that when $\omega(q_1)\equiv 1,$ the
system reduces to a  highly oscillatory Hamiltonian system with a
constant-frequency, which has been studied by using the technique of
modulated Fourier expansions in \cite{Hairer00,hairer2006}. In the
recent two decades, modulated Fourier expansions have been developed
as an important mathematical tool in the study of  the long-time
behaviour for various numerical methods and different differential
equations (see, e.g.
\cite{Cohen06,Cohen15,Cohen05,Cohen08-1,Faou14,Gauckler17,Gauckler13,Hairer00,Hairer08,Hairer09,Hairer12-1,Hairer13,iserles08,McLachlan14,Sanz-Serna09,Stern09}).
Very  recently, Hairer and Lubich in \cite{Hairer16} developed a
varying-frequency modulated Fourier expansion for the highly
oscillatory Hamiltonian system \eqref{H}, and then  the long-time
analysis of the St\"{o}rmer--Verlet (SV) method was presented based
on the Fourier expansion.

On the other hand,  in order to solve highly oscillatory systems
effectively and efficiently,   a standard form of extended
Runge--Kutta--Nystr\"{o}m (ERKN) integrators was formulated in
\cite{wu2010-1}, and  in the sense of geometric numerical
integration they have been researched (see, e.g.
\cite{CiCP(2017)_Mei_Liu_Wu,wang-2016,wang2018-IMA,wu2017-JCAM,wubook2015,xinyuanbook2018,wu2012,wu2013-book}).
Recently, the long-term energy behaviour  of ERKN integrators for
highly oscillatory Hamiltonian systems with constant frequencies has
been researched in  \cite{17-new} by the modulated Fourier
expansion. \emph{However, we note a fact that long-time behaviour of
trigonometric integrators for the highly oscillatory system
\eqref{H} with a solution-dependent high frequency has not been
considered and researched so far}.

When dealing with the nonlinear highly oscillatory system
$\ddot{q}+\Omega^2q=f(q)$ such as the well-known Fermi--Pasta--Ulam
problem of Section I.5.1 in \cite{hairer2006}, the ERKN integrator
behaves much better than  the SV method, not only in the accuracy
but also in the energy preservation (see Figure \ref{P5-2} for the
results obtained by a one-stage ERKN integrator which will be
researched in this paper, and the SV method, respectively). A
remarkable advantage can be observed from an ERKN integrator that
the corresponding highly oscillatory linear homogeneous
 equation $\ddot{q}+\Omega^2q=0$ can be exactly solved by both the internal stages and updates of an  ERKN integrator.
 Unfortunately, however, the SV method cannot share this important property.
 This point is crucial since the high oscillation is brought by the
linear part $\Omega^2q$ (see, e.g.
\cite{CiCP(2017)_Mei_Liu_Wu,xinyuanbook2018} for details).
%{\color{blue}Moreover, we use the well known Fermi--Pasta--Ulam
%problem of Section I.5.1 in \cite{hairer2006} to show the numerical
%results (see Figure \ref{P5-2}) of an ERKN integrator (the one
%researched in this paper) and the SV method.
\begin{figure}[!htb]
\centering
\begin{tabular}[c]{cccc}%
  % Requires \usepackage{graphicx}
\includegraphics[width=4cm,height=4.5cm]{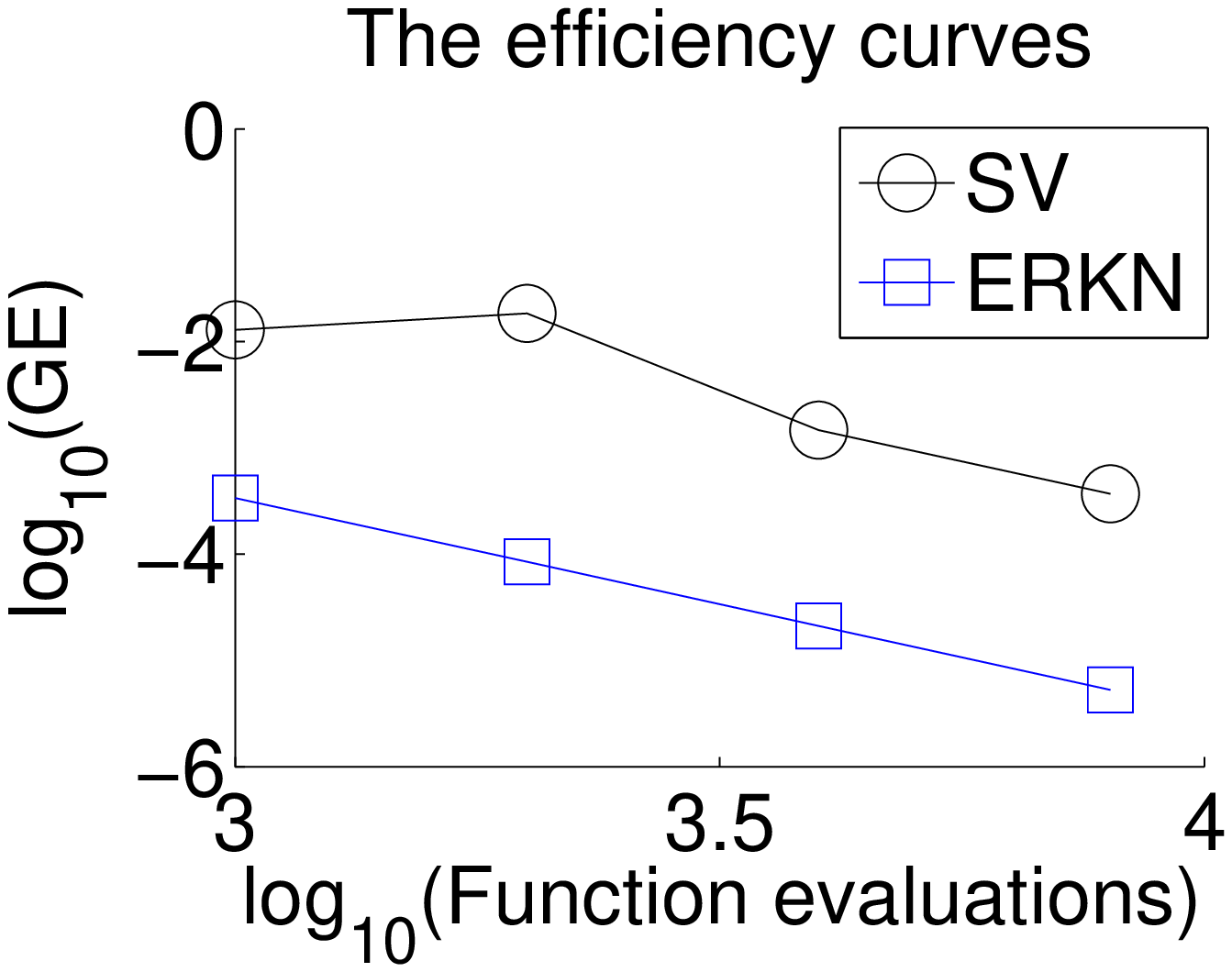}
\includegraphics[width=4cm,height=4.5cm]{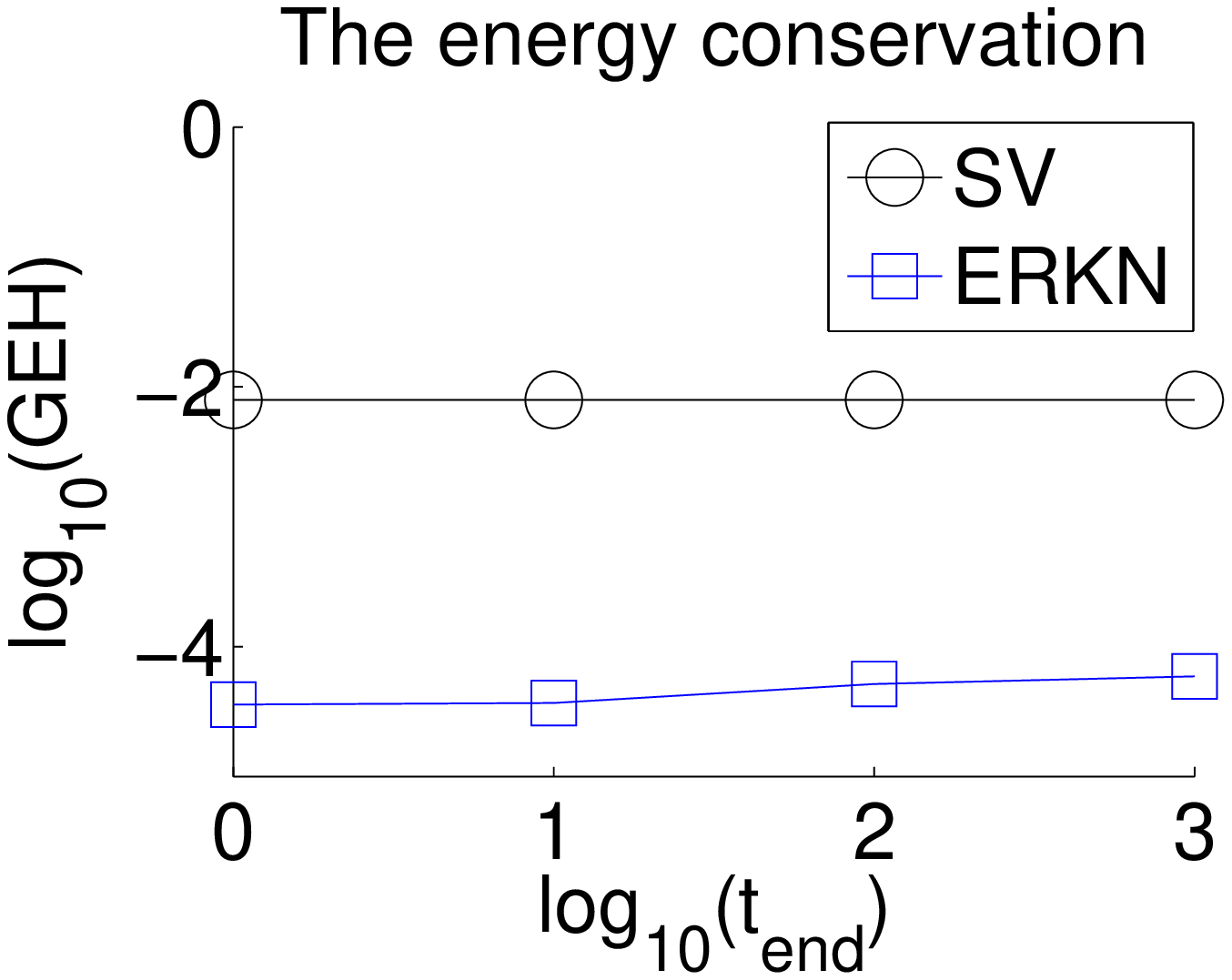}
\end{tabular}
\caption{Left figure: The logarithm of the  global errors against
the logarithm of number of function evaluations  (the problem is
solved in $[0,10]$ with $h=0.01/k$ for $k=0,\ldots,3$). Right
figure: The logarithm of the maximum  global energy errors against
the logarithm of $t_{\textmd{end}}$  (it is solved with $h=0.005$ in
$[0,t_{\textmd{end}}]$ with $t_{\textmd{end}}=10^k$ for
$k=0,\ldots,3$).}\label{P5-2}
\end{figure}

%It can be observed clearly that the ERKN integrator behaves much
%better than the SV method not only in the accuracy but also in the
%energy preservation.

The fact stated above  motivates us to further study the long time
behaviour of ERKN integrators
%when applied to  the highly oscillatory system \eqref{H}.}
%Under the background mentioned above, this paper is devoted to the  long-time behaviour of ERKN integrators
when applied to highly oscillatory Hamiltonian systems with a slowly
varying, solution-dependent high frequency. Using the technique of
varying-frequency modulated Fourier expansion developed in
\cite{Hairer16}, we will present a varying-frequency modulated
Fourier expansion for one-stage explicit ERKN integrators and
analyse their long-time behaviour. \emph{We remark that this paper
seems to be the first research in the literature that rigorously
studies the  long-time behaviour of trigonometric integrators on
highly oscillatory systems with a solution-dependent high
frequency.} As a byproduct  of this paper, we also show the
long-time behaviour of an RKN method, and it turns out that this RKN
method conserves the same modified action and modified energy as
those given in
  \cite{Hairer16}.
  Moreover, the analysis presented in this paper   is suitable for
the  long-time behaviour of the trigonometric integrators given in
Chap. XIII of \cite{hairer2006} and the implicit-explicit methods
considered in \cite{McLachlan14,Stern09} from the case of a constant
high frequency to a solution-dependent high frequency.

 For the long-time
 behaviour of numerical methods applied to system \eqref{H}, according to the comments ``... as a concise proof of
concept, we do not work out these interesting extensions here, but
limit ourselves to the exemplary case of the St\"{o}rmer--Verlet
method" quoted from Harier and Lubich in \cite{Hairer16}, the
analysis of ERKN integrators is more complicated than that of SV
method since ERKN integrators need the constant frequency in
computations. We note that this paper will present ``\emph{these
interesting extensions}" but is not aimed at the superiority of ERKN
integrator compared with the SV method. This present paper tries to
show that the ERKN integrator approximately conserves a modified
action and a modified total energy over long time intervals by using
the technique of varying-frequency modulated Fourier expansion.
Moreover,  we try to get similar long-time behaviour  for an RKN
method  in a simple way, which is yielded as a byproduct of this
paper. In the part of numerical examples, we will show the numerical
behaviour for the ERKN integrator  and SV method. It will be seen
that these two methods have good long-time behaviour and the ERKN
integrator behaves better than the SV method.

It is noted that in comparison with highly oscillatory problems with
constant frequencies, the solution of Hamiltonian systems with a
solution-dependent high frequency admits a varying-frequency
modulated Fourier expansion that extends the constant-frequency
modulated Fourier expansion of \cite{Hairer00,hairer2006}. The novel
expansion was developed by Hairer and Lubich in \cite{Hairer16} and
was   used to the SV method to obtain its long-time behaviour. For
the analysis of ERKN integrators, this new technology will also be
applied  instead of the modulated Fourier expansion which was
considered in \cite{17-new}.

The rest of this paper is organised as follows.  In Section
\ref{sec:Formulation},  we present the scheme of ERKN integrators
for the solution of  \eqref{H-s}. Then the main results of this
paper are given in Section \ref{sec:main results}, and a numerical
experiment is reported in Section \ref{sec:examples}. In order to
prove the main results, the
 varying-frequency modulated Fourier expansion of  an ERKN
integrator is established rigorously in Section  \ref{sec:Analysis
of the methods}. In Section \ref{sec:Almost-invariants}, we show
that the integrator nearly conserves a modified action over a long
term. A similar result about the  near  conservation of  a modified
energy over   a long term is derived in Section
\ref{sec:energy-invariants}. Some concluding remarks are drawn in
Section \ref{sec:conclusions}.

\section{ERKN integrators} \label{sec:Formulation}
In order to apply an ERKN integrator to the highly oscillatory
Hamiltonian system \eqref{H}, we rewrite  \eqref{H}  as the
following system of second-order ordinary differential equations
(ODEs)
 \begin{equation}
\begin{array}
[c]{ll}%
\ddot{q}+\Omega^2q =g(q):=\left(
                            \begin{array}{c}
                              -\frac{\omega(q_1) |q_2|^2}{\epsilon^2}\nabla_{q_1}\omega(q_1)-\nabla_{q_1}U(q) \\
                              -\frac{\omega(q_1)^2-\omega_0^2}{\epsilon^2} q_2-\nabla_{q_2}U(q) \\
                            \end{array}
                          \right),
\end{array}
  \label{prob}%
\end{equation}
where $\omega_0=\omega(q_1(0)),\
\upsilon=\frac{\omega_0}{\epsilon},$ $$\Omega= \left(
                          \begin{array}{cc}
                            0_{d_1\times d_1}  & 0 _{d_2\times d_2}  \\
                            0_{d_1\times d_1} &  \upsilon I_{d_2\times d_2} \\
                          \end{array}
                        \right) ,$$
 and $g$ is   the negative
gradient of the following   real-valued function $W$:
\begin{equation*}W(q)=\frac{\omega(q_1)^2-\omega_0^2}{2\epsilon^2}
|q_2|^2+U(q).
\label{Wq}%
\end{equation*}

\begin{rem}
It is noted that this treatment replaces  a factor
$-\frac{\omega(q_1)^2}{\epsilon^2} q_2$  in the nonlinearity for the
SV method by the presence of
$-\frac{\omega(q_1)^2-\omega_0^2}{\epsilon^2} q_2$ in the second
component of $g(q)$. The bound of this new term is smaller since
 the high frequency
$\frac{\omega(q_1)}{\epsilon}$  varies  slowly. Moreover, the new
presence does not affect the estimates of \eqref{coefficient func1}
compared to the SV method. Therefore, it will not  cause any
problems in the following analysis.
\end{rem}

 A   standard form of ERKN integrators was   formulated   in \cite{wu2010-1}.  Here, we
 represent its one-stage explicit scheme in a compact and
 conventional form.

\begin{defi}
\label{erkn}  (See \cite{wu2010-1}) A one-stage  explicit ERKN
integrator for solving (\ref{prob}) is defined by%
 \begin{equation}
\begin{array}
[c]{ll}%
Q^{n+c_{1}} &
=\cos(c_{1}h\Omega)q^{n}+hc_{1}\mathrm{sinc}(c_{1}h\Omega)p^{n},\\
q^{n+1} & =\cos(h\Omega)q^{n}+h\mathrm{sinc}(h\Omega)p^{n}+h^{2}
 \bar{b}_{1}(h\Omega)g(Q^{n+c_{1}}),\\
p^{n+1} &
=-h\Omega^2\mathrm{sinc}(h\Omega)q^{n}+\cos(h\Omega)p^{n}+h\textstyle
b_{1}(h\Omega)g(Q^{n+c_{1}}),
\end{array}
  \label{erknSchems}%
\end{equation}
where   $h$ is  a  stepsize,  $c_1$ is a real constant  and
$c_1\in[0,1]$, $b_{1}(h\Omega)$ and $\bar{b}_{1}(h\Omega)$ are
matrix-valued and bounded functions. Here $\mathrm{sinc}(h\Omega)$
is defined by $(h\Omega)^{-1}\sin(h\Omega)$.
\end{defi}

 According to the  symmetry conditions of the integrators  derived in \cite{wu2013-book},
the ERKN integrator \eqref{erknSchems} is symmetric if and only if
\begin{equation}\begin{aligned}\label{sym cond}&c_1=1/2,\quad  \bar{b}_1(h\Omega)=(h\Omega(I+\cos(h\Omega)))^{-1}{\sin(h\Omega)}b_1(h\Omega).
\end{aligned}\end{equation}

\begin{rem} Concerning the computational cost of ERKN
integrators compared with the SV method, the main part of the
computational cost consists in the
 nonlinearity since the
matrix-valued functions need to be computed only once. Both
one-stage ERKN integrators and the SV method require the same number
of nonlinear function evaluations.
\end{rem}

\begin{rem}In order to  present this
paper as a concise proof of concept, we only consider one ERKN
integrator in the remainder   of the analysis. It is noted that with
the techniques used in this paper and by some modifications if
necessary, it is feasible to derive the long-time results for other
one-stage ERKN integrators.
\end{rem}

%
%\begin{theo}\label{symplectic thm}
%If there exists a real number $d_1$ such that
%\begin{equation}\begin{aligned}\label{symple cond}&
%b_1(h\upsilon)=d_1\cos((1-c_1)h\upsilon),\
% \ \ \bar{b}_1(h\upsilon)=d_1\frac{\sin((1-c_1)h\upsilon)}{h\upsilon},
%\end{aligned}\end{equation}
%then the ERKN integrator \eqref{erknSchems} is symplectic.
%\end{theo}

\section{Main results} \label{sec:main results}
\subsection{Modified action conservation of an ERKN integrator}
 This part gives the modified action conservation of an ERKN
integrator. We begin with introducing the modified frequency
\begin{equation*}
\omega_h(q_1(t))= \omega(q_1(t))\sqrt{1-\frac{h^2}{4\epsilon^2}
\mathrm{sinc}^2(h\upsilon/2) \omega^2(q_1(t))}
%=\mathrm{sinc}(h\upsilon/2)
%\omega(q_1(t))\sqrt{1-\frac{h^2}{4\epsilon^2}
%\mathrm{sinc}^2(h\upsilon/2) \omega^2(q_1(t))}
\label{modi fre}%
\end{equation*}
and the following modified action
\begin{equation}\begin{aligned}
I_h(q,p)%:=&\frac{1}{2}\mathrm{sinc}^2(h\upsilon)\cos(\frac{1}{2}h\upsilon)(
%\bar{b}_1(h\upsilon))^{-1} \frac{|p_2|^2}{\omega_h(q_1)}
%+\cos(\frac{1}{2}h\upsilon)(
%\bar{b}_1(h\upsilon))^{-1}\frac{\omega_h(q_1)}{2\epsilon^2}|q_2|^2\\
=& \frac{\Psi(h\upsilon,q_1)\mathrm{sinc}^2
(h\upsilon)}{2\mathrm{sinc} (h\upsilon/2)}
  \frac{|p_2|^2}{ 2\omega_h(q_1)} + \frac{\Psi(h\upsilon,q_1) \mathrm{sinc}
(h\upsilon/2)}{2} \frac{ \omega_h(q_1)}{2\epsilon^2}|q_2|^2,
\label{modi act}%
\end{aligned}
\end{equation}
where $$
  \Psi(h\upsilon,q_1)=\cos(\frac{1}{2}h\upsilon)(
\bar{b}_1(h\upsilon))^{-1} +
\frac{1}{2}h^2\upsilon^2\mathrm{sinc}(\frac{1}{2}h\upsilon)(
b_1(h\upsilon))^{-1}\Big(\frac{\mathrm{sinc}^2(h\upsilon/2)}{\mathrm{sinc}^2(h\upsilon)}\frac{\omega_h^2(q_1)}{
\omega_0^2}\Big) .$$ It can be verified that when $\upsilon=0$,
these modified frequency and action become    the modified frequency
and action given by Hairer and Lubich in \cite{Hairer16},
respectively.

\begin{theo}\label{pre mod I} \textbf{(Modified action
conservation.)}
Choose a stepsize $h=\mathcal{O}(\epsilon)$ %%%%%%%%%%%%%%%%%%%%%%%%%%%%%%%%%%%%%%%%%%%%%%%%%%%%%%%%%%%%%%%%%%%%%%%%
\footnote{{  This condition    is used such that assumption
\eqref{numerical cond}  can be true and suitable  bounds of
coefficients functions appearing in the modulated Fourier expansion
  can be obtained.  The same requirement is also used in  \cite{Hairer16}. }}
%%%%%%%%%%%%%%%%%%%%%%%%%%%%%%%%%%%%%%%%%%%%%%%%%%%%%%%%%%%%%%%%%%%%%%%%
and consider the   ERKN integrator \eqref{erknSchems} for solving
\eqref{prob} with initial values \eqref{Initial val}. It is assumed
that for $0\leq nh \leq T$, $q^n=(q_1^n,q_2^n)$ stays in a compact
set $K$ and %%%%%%%%%%%%%%%%%%%%%%%%%%%%%%%%%%%%%%%%%%%%%%%%%%%%%%%%%%%%%%%%%%%%%%%%
\footnote{   It is noted that   the condition \eqref{numerical cond}
is an artificial assumption such that the proof of Proposition
\ref{energy thm} works. A similar assumption (formula (4.4)) has
been given in \cite{Hairer16} and when $\upsilon=0$, the condition
\eqref{numerical cond}  becomes the assumption (4.4) given by Hairer
and Lubich in \cite{Hairer16}.}
%%%%%%%%%%%%%%%%%%%%%%%%%%%%%%%%%%%%%%%%%%%%%%%%%%%%%%%%%%%%%%%%%%%%%%%%
\begin{equation}
\frac{h}{\epsilon}  \mathrm{sinc}(h\upsilon/2)
\omega(q_{1}^{n})\leq2\sin(\frac{\pi}{N+2})
\label{numerical cond}%
\end{equation}
  for some odd integer $N\geq1.$
Moreover, we  require the symmetry conditions \eqref{sym cond} and
choose $$\bar{b}_1(h\upsilon)= \frac{1}{2}
\mathrm{sinc}^2(h\upsilon/2).$$ In this paper, the stepsize $h$ is
chosen such that the coefficients $\bar{b}_1(h\upsilon),
b_1(h\upsilon)$ of ERKN integrators are not zero. We have the
following long-term modified action conservation of the ERKN
integrator:
\begin{equation*}
I_h(q_n,p_n)=I_h(q_0,p_0)+\mathcal{O}(\epsilon)
\end{equation*}
for $0\leq n h\leq \epsilon^{-N+1},$ where the constant  symbolized
by $\mathcal{O}$ depends on $N$ but  is independent of $n, h$ and $
\epsilon$.
\end{theo}

\subsection{Modified energy conservation of an ERKN integrator}
 A modified energy conservation of an ERKN integrator will be presented in this part. To do this, we introduce another modified frequency
\begin{equation*}
\tilde{\omega}_h(q_1)= \frac{2\epsilon}{h} \mathrm{arcsin}\big(
\frac{h}{2\epsilon} \mathrm{sinc}(h\upsilon/2) \omega(q_1)\big)
\label{modi fre ano}%
\end{equation*}
and the following modified energy
\begin{equation}\begin{aligned}
H_h(q,p)%:=&\frac{1}{2}\mathrm{sinc}^2(h\upsilon)\cos(\frac{1}{2}h\upsilon)(
%\bar{b}_1(h\upsilon))^{-1} \frac{|p_2|^2}{\omega_h(q_1)}
%+\cos(\frac{1}{2}h\upsilon)(
%\bar{b}_1(h\upsilon))^{-1}\frac{\omega_h(q_1)}{2\epsilon^2}|q_2|^2\\
=&\frac{1}{2}|p_1|^2+   \tilde{\omega}_h(q_1) I_h(q,p) + U(q)+ \Big(
1-\Psi(h\upsilon,q_1) \bar{b}_1(h\upsilon) \Big) \frac{
\omega^2(q_1)-\omega_0^2}{\epsilon^2} |q_2|^2.
\label{modi ene}%
\end{aligned}
\end{equation}
It is noted that  when $\upsilon=0$, these modified frequency and
modified energy reduce to those introduced by  Hairer and Lubich in
\cite{Hairer16}.

\begin{theo}\label{pre mod H} \textbf{(Modified energy
conservation.)} Under the conditions of Theorem \ref{pre mod I}, the
following long-term modified energy conservation holds
\begin{equation*}
H_h(q_n,p_n)=H_h(q_0,p_0)+\mathcal{O}(\epsilon)
\end{equation*}
for $0\leq n h\leq \epsilon^{-N+1},$ where the constant  symbolized
by $\mathcal{O}$ depends on $N$ but  is independent of $n,\ h$ and $
\epsilon$.
\end{theo}

From Section \ref{sec:Analysis of the methods} to Section
\ref{sec:energy-invariants}, we will  prove the long-time
 behaviour of ERKN integrators stated in Theorems \ref{pre mod I} and \ref{pre mod H}.  For the analysis, the following key points will be considered one
by one.
  \begin{itemize}
 \item In Section \ref{sec:Analysis of the methods}   the varying-frequency
modulated Fourier expansion for ERKN integrators  is derived and the
bounds of the coefficient functions and defects are estimated.

\item In Section \ref{sec:Almost-invariants},  an invariant
$\mathcal{I}$ of the modulated Fourier expansion is shown  by
Proposition \ref{second invariant thm} and its relationship with the
modified actions is established. On the basis of these analyses,
Theorem \ref{pre mod I} is proved.

\item In Section \ref{sec:energy-invariants}, Proposition \ref{H invariant thm} shows another
invariant $\mathcal{H}$ of the modulated Fourier expansion. The
relationship between this invariant  and the modified  energy is
also established in this theorem. Finally,  the long-time
near-conservation of the modified energy $H_h$ stated in Theorem
\ref{pre mod H} is proved.
\end{itemize}

 It is remarked that the above  procedure is  a standard approach to the study of the
long-time behavior for numerical methods of Hamiltonian ordinary
differential equations by using modulated Fourier expansions (see,
e.g. \cite{Cohen06,Cohen15,Cohen05,Hairer00,Hairer09,Hairer16}). The
detailed proof of each key point   closely follows \cite{Hairer16}
 but with some necessary modifications adapted to the ERKN integrator.
 For brevity, we will only present the differences and skip the same derivations as those of \cite{Hairer16}.

%%%%%%%%%%%%%%%%%%%%%%%%%%%%%%%%%%%%%%%%%%%%%%%%%%%%%%%%%%%%%%%%%%%%%%%%%%%%%%%%%%
\subsection{Long time conservation of an RKN method}  It
is noted that when  $\upsilon=0$, the ERKN integrator \label{eqref}
for solving (\ref{prob}) reduces to the following one-stage explicit
RKN method  for solving (\ref{H})
 \begin{equation}
\begin{array}
[c]{ll}%
Q^{n+c_{1}} &
= q^{n}+hc_{1} p^{n},\\
q^{n+1} & = q^{n}+h p^{n}+h^{2}
 \bar{b}_{1} g(Q^{n+c_{1}}),\\
p^{n+1} & =p^{n}+h\textstyle b_{1}g(Q^{n+c_{1}}),
\end{array}
  \label{RKN}%
\end{equation}
where  $g$ is the nonlinearity in  \eqref{prob} without $\omega_0$.
%$$g=\left(
%                            \begin{array}{c}
%                              -\frac{\omega(q_1) |q_2|^2}{\epsilon^2}\nabla_{q_1}\omega(q_1)-\nabla_{q_1}U(q) \\
%                              -\frac{\omega(q_1)^2}{\epsilon^2} q_2-\nabla_{q_2}U(q) \\
%                            \end{array}
%                          \right).$$
All the analysis presented in this paper holds true for   this
special case. Accordingly, we obtain the following result for the
RKN method \eqref{RKN}.

\begin{cor}\label{pre RKN}
\textbf{(Long time conservations of an RKN method.)} Choose a
stepsize $h=\mathcal{O}(\epsilon)$ and assume that   the numerical
solution of the   RKN method \eqref{RKN}  stays in a compact set.
Under the conditions $c_1=\frac{1}{2},  \bar{b}_1= \frac{1}{2}, b_1=
1$ and $ \frac{h}{\epsilon}
\omega(q_{1}^{n})\leq2\sin(\frac{\pi}{N+2}) $
  for some odd integer $N\geq1,$  we have the following long time conservations of the RKN
  method \eqref{RKN}
\begin{equation*}\begin{array}
[c]{ll}
\hat{I}_h(q_n,p_n)=\hat{I}_h(q_0,p_0)+\mathcal{O}(\epsilon),\ \
\hat{H}_h(q_n,p_n)=\hat{H}_h(q_0,p_0)+\mathcal{O}(\epsilon)\end{array}
\end{equation*}
for $0\leq n h\leq \epsilon^{-N+1},$ where the constants   depend on
$N$ but  are independent of $n,\ h$ and $ \epsilon$. Here the
modified action $\hat{I}_h$ and modified energy $\hat{H}_h$ are
given  respectively by:
\begin{equation*}\begin{array}[c]{ll}
&\hat{I}_h(q,p)=    \frac{|p_2|^2}{ 2\omega_h(q_1)} + \frac{
\omega_h(q_1)}{2\epsilon^2}|q_2|^2\ \ \mathrm{with} \ \omega_h(q_1)=
\omega(q_1(t))\sqrt{1-\frac{h^2}{4\epsilon^2} \omega^2(q_1(t))},\\
&\hat{H}_h(q,p)=\frac{1}{2}|p_1|^2+   \frac{2\epsilon}{h}
\mathrm{arcsin}\big( \frac{h}{2\epsilon}  \omega(q_1)\big)
\hat{I}_h(q,p) + U(q),
\end{array}\end{equation*}
which  are  the same as those given by  Hairer and Lubich in
\cite{Hairer16}.
\end{cor}

\section{A numerical experiment} \label{sec:examples}

As an illustrative example, we consider the Fermi--Pasta--Ulam
problem of Section I.5.1 in \cite{hairer2006}, where we replace the
constant frequency by a slowly varying, solution-dependent high
frequency (see \cite{Hairer16}). The system can be expressed by
\eqref{H} with $\omega(q_1)=1+\sin^2(q_{11})$ ($q_{11}$ denotes the
first entry of the vector $q_{1}$) and
$$U(q)=\dfrac{1}{4}[(q_{11}-q_{21})^{4}
+ (q_{12}-q_{13}-q_{11}-q_{21} )^{4}+  (q_{13}-q_{21}-q_{12}-q_{22}
)^{4}+(q_{13}+q_{23})^{4}].
$$
Following \cite{Hairer16,hairer2006}, we choose $\epsilon=0.01$ and
\begin{equation*}\label{initial date} \ q_{11}(0)=1,\ p_{11}(0)=1,\
q_{21}(0)=\dfrac{1}{\omega},\ p_{21}(0)=1,
 \end{equation*}
with zero for the remaining initial values.  The coefficients of the
one-stage explicit ERKN integrator are given by
$$c_1=\frac{1}{2}, \quad \bar{b}_1(h\upsilon)= \frac{1}{2}
\mathrm{sinc}^2(h\upsilon/2), \quad b_1(h\upsilon)=
\cos(h\upsilon/2)\mathrm{sinc}(h\upsilon/2).$$

\begin{figure}
% Use the relevant command to insert your figure file.
% For example, with the graphicx package use
\includegraphics[width=12cm,height=2.5cm]{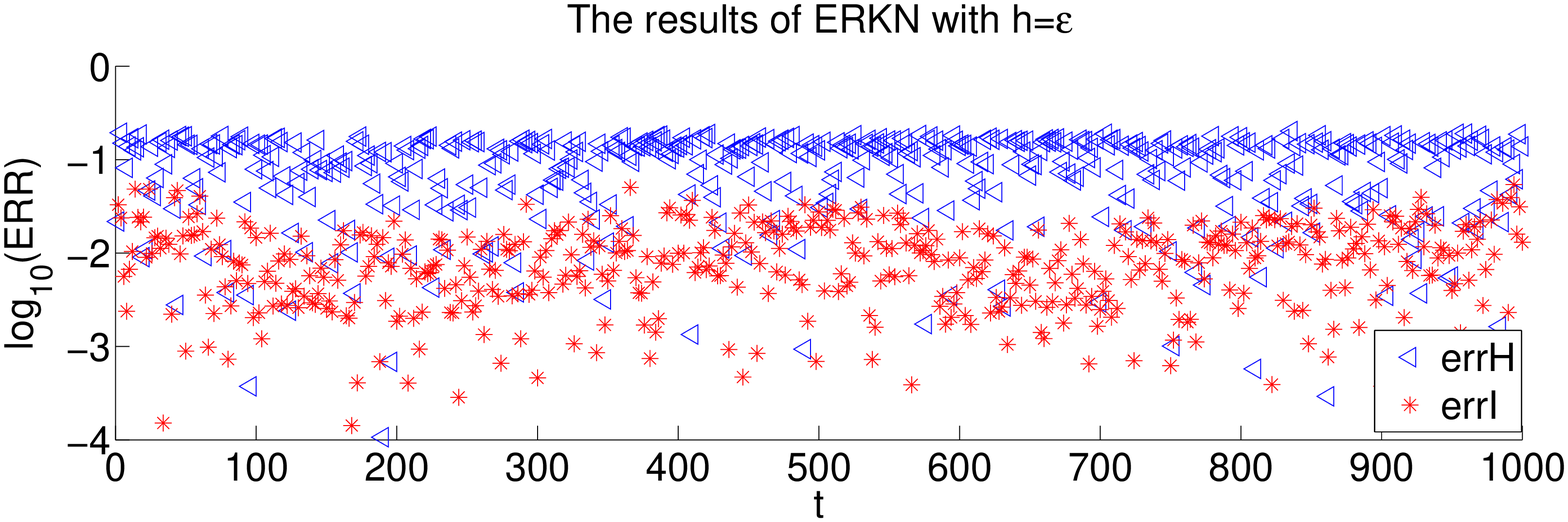}\\
\includegraphics[width=12cm,height=2.5cm]{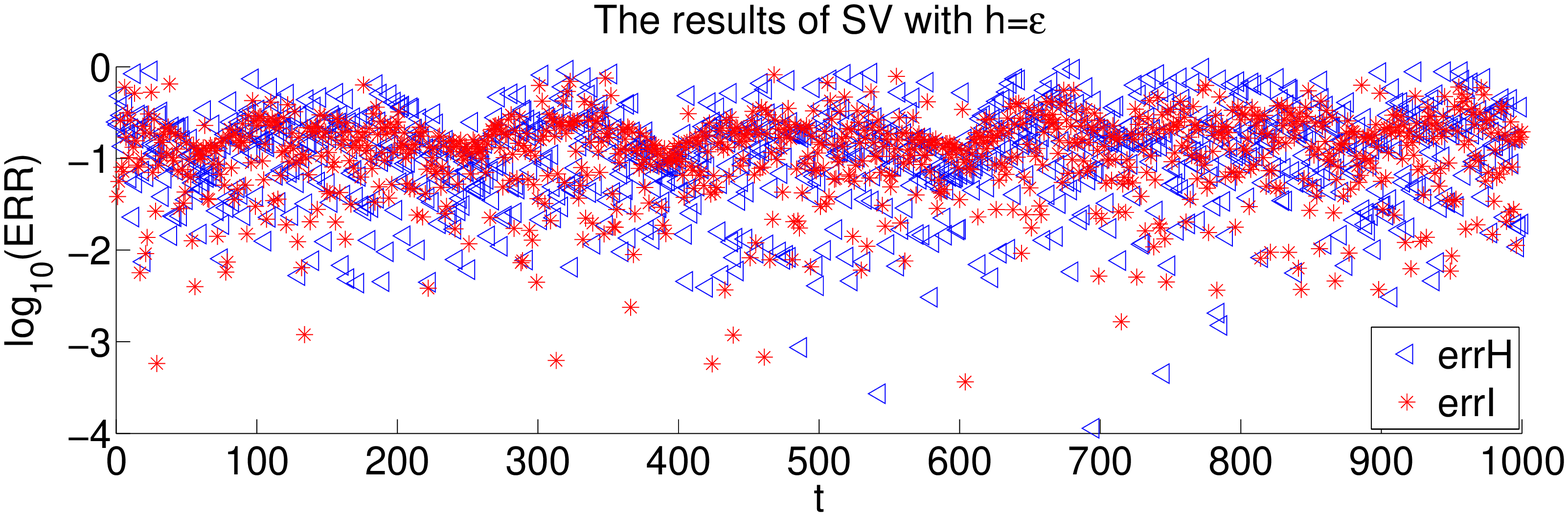}\\
\includegraphics[width=12cm,height=2.5cm]{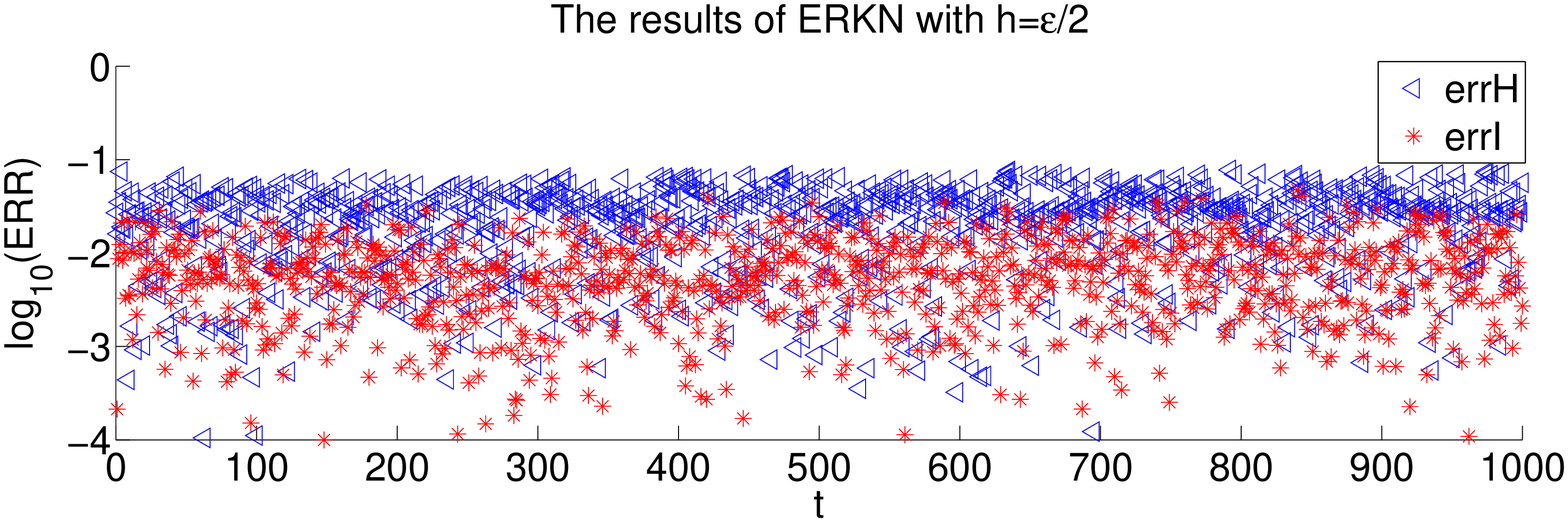}\\
\includegraphics[width=12cm,height=2.5cm]{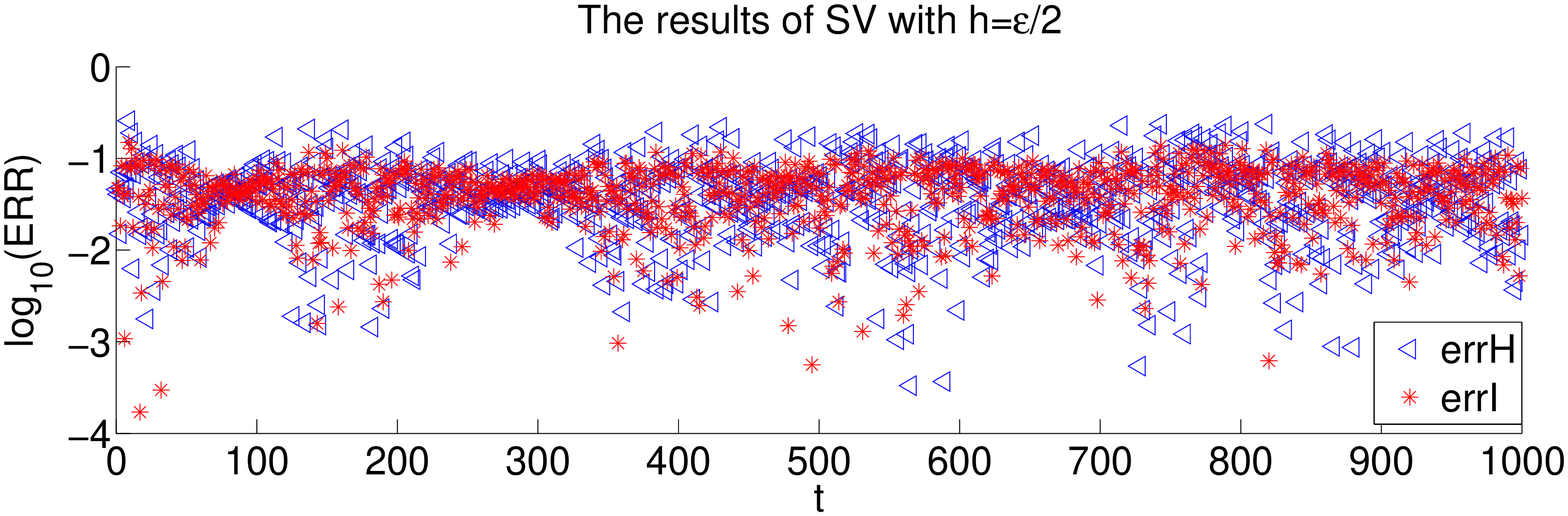}\\
\includegraphics[width=12cm,height=2.5cm]{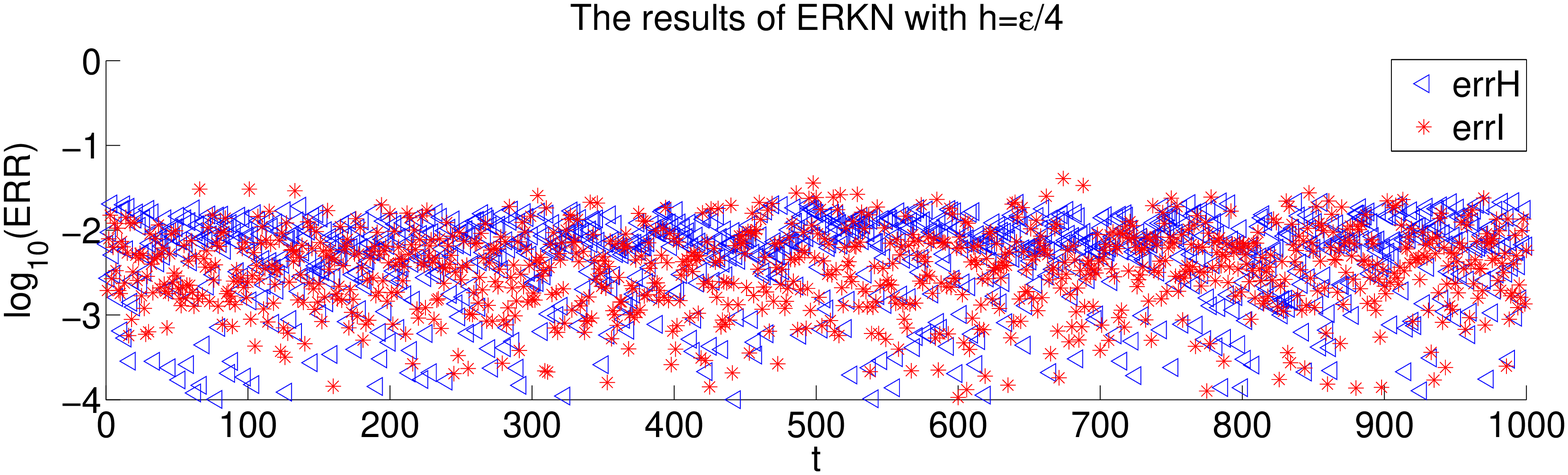}\\
\includegraphics[width=12cm,height=2.5cm]{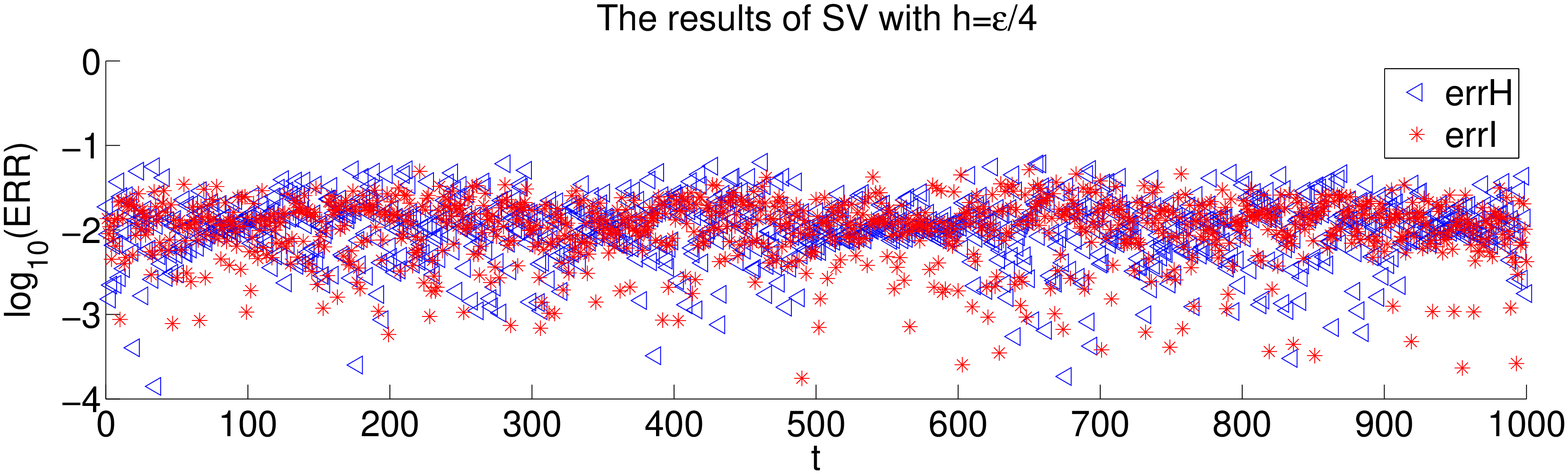}
% figure caption is below the figure
\caption{The logarithm of the  errors   of  $H_h$ and   $I_h$ against $t$.}%
\label{fig0}%
\end{figure}

 We solve the system on the interval $[0,1000]$ with $h=\epsilon,
\epsilon/2$ and $\epsilon/4$.   The errors of the modified action
$I_h$ and modified energy $H_h$ against $t$ for the ERKN integrator
and for the SV method are shown in Fig. \ref{fig0}, respectively. It
can be observed from Fig. \ref{fig0} that these two methods have
good long-time behaviour. Moreover, in order to compare the
performance of these two methods in a clearer way, we present the
errors (logarithmic scale) of $H_h$ and $I_h$ in Tables \ref{ta
re1}-\ref{ta re2} for $h=\epsilon, \epsilon/2$ and $\epsilon/4$. It
can be seen from Tables \ref{ta re1}-\ref{ta re2} that  the ERKN
integrator behaves better than the SV method.

\renewcommand\arraystretch{1.2}
\begin{table}[!htb]$$
\begin{array}{|c|c c|c c|c c|}
\hline%%%%%%%%%%%%%%%%%%%%%%%%%%%%%%%%%%%
  &h=\epsilon   &   &h=\frac{\epsilon}{2}   & &h=\frac{\epsilon}{4}   &\\
t &ERKN   &SV  &ERKN    &SV  &ERKN   &SV  \\
\hline%%%%%%%%%%%%%%%%%%%%%%%%%%%%%%%%%%%
%%%%%%%%%%%%%%%%%%%%%%%%%%%%%%%%%%%%%%%%%%%%%%%%%%%%%%%%%%%%%%%%
100   &-1.8344 & -0.6452  &-2.5332  & -1.4777 &-2.6127 &-1.8497    \cr%%%%%%%%%%%%
200  & -2.7308 & -0.6029  &-2.3266 & -1.6145 & -1.8030 &-1.8926  \cr%%%%%%%%%%%%
300  & -3.3372 & -1.4396  &-2.5646 & -1.3632 & -2.8197 &-1.7694 \cr%%%%%%%%%%%%
400  & -2.0022 & -0.8373  &-2.9703 & -1.1985&  -2.0974 &-1.7005 \cr%%%%%%%%%%%%
500  & -1.5597 &-0.8239   &-1.6116 &-1.7591 &  -1.6143 & -1.6400 \cr%%%%%%%%%%%%
600  & -2.0526 & -1.4494  &-2.2234 & -1.0829& -3.9740  &-2.7206\cr%%%%%%%%%%%%
700  & -2.6053 & -0.5843  &-2.1349 & -1.4267 & -3.4172 &-1.4197 \cr%%%%%%%%%%%%
800  & -1.6701 &-0.5793   &-1.9155 &-1.0739 & -2.2083  &-1.9014\cr%%%%%%%%%%%%
900  & -3.4265 & -1.2501  &-2.3435 & -1.4888& -2.6601  &-2.2565\cr%%%%%%%%%%%%
 \hline
\end{array}
$$
\caption{The logarithm of the  errors   of   $I_h$  at different $t$
for different $h$.} \label{ta re1}
\end{table}

\renewcommand\arraystretch{1.2}
\begin{table}[!htb]$$
\begin{array}{|c|c c|c c|c c|}
\hline%%%%%%%%%%%%%%%%%%%%%%%%%%%%%%%%%%%
  &h=\epsilon   &   &h=\frac{\epsilon}{2}   & &h=\frac{\epsilon}{4}   &\\
t &ERKN   &SV  &ERKN    &SV  &ERKN   &SV  \\
\hline%%%%%%%%%%%%%%%%%%%%%%%%%%%%%%%%%%%
%%%%%%%%%%%%%%%%%%%%%%%%%%%%%%%%%%%%%%%%%%%%%%%%%%%%%%%%%%%%%%%%
100   &-1.1034  & -1.4116 &  -3.9534  & -1.5614  & -2.0679 &-1.8594
\cr 200   &-1.0398  & -1.0472 &  -2.0653 &  -1.9351  & -3.3222  &
-2.1970 \cr 300    &-1.3770  & -0.7537  & -1.4473 &  -1.3941 &
-2.0671 & -1.8029 \cr
  400 &-0.8716  & -1.2544 &  -1.6365  & -1.1100 &  -2.0744  & -1.6724 \cr
  500& -0.9011  & -1.2598  & -2.2577  & -1.9558  & -1.7846  & -2.2547  \cr
  600 &-0.8535  & -1.7726 &  -3.1624  & -0.9785 &  -1.9745  & -1.9599 \cr
   700&-2.5285  & -0.7339  & -1.7042 &  -1.7082  & -2.1943  & -2.2916 \cr
  800& -0.7612  & -0.2928 &  -1.6261  & -0.8134  & -3.8582 &  -1.8712 \cr
  900 &-1.3052  & -0.8408  & -1.5555 &  -1.6918  & -2.2099  & -1.7959
 \cr
 \hline
\end{array}
$$
\caption{The logarithm of the  errors   of   $H_h$  at different $t$
for different $h$.} \label{ta re2}
\end{table}

\section{Modulated Fourier expansion for the integrators} \label{sec:Analysis of the methods}

In  what follows,  we derive and analyse  the varying-frequency
modulated Fourier expansion for ERKN integrators.

\begin{prop}\label{energy thm}
Under the conditions of Theorem \ref{pre mod I},  the numerical
solution determined by \eqref{erknSchems} admits the following
varying-frequency modulated Fourier expansion
\begin{equation}
\begin{aligned} &q^{n}= \sum\limits_{|k|\leq N+1} \mathrm{e}^{\mathrm{i}k\phi(t)/\epsilon}\zeta_h^k(t)+R_{h,N}(t),\\
&p^{n}= \sum\limits_{|k|\leq N+1} \mathrm{e}^{\mathrm{i}k\phi(t)/\epsilon}\eta_h^k(t)+S_{h,N}(t),\\
\end{aligned}
\label{MFE-ERKN}%
\end{equation}
where $\zeta_{h,2}^{N+1}(t)=\eta_{h,2}^{N+1}(t)=0,$ the phase
function $\phi(t)$ satisfies
%%%%%%%%%%%%%%%%%%%%%%%%%%%%%%%%%%%%%%%%%%%%%%%%%%%%%%%%%%%%%%%%%%%%%%%%
\footnote{{  Clearly, it follows from the assumption, i.e.
$\bar{b}_1(h\upsilon)\neq 0$ that $\textmd{sinc}(hv/2) \neq 0$. This
guarantees that the function $\phi(t)$ is not constant. }}
%%%%%%%%%%%%%%%%%%%%%%%%%%%%%%%%%%%%%%%%%%%%%%%%%%%%%%%%%%%%%%%%%%%%%%%%
\begin{equation}
 \sin(\frac{h\dot{\phi}(t)}{2\epsilon}
 )=\frac{h}{2\epsilon}
\mathrm{sinc}(h\upsilon/2) \omega(\zeta_{h,1}^{0}(t)),\ \ \
\phi(0)=0,
\label{phase func}%
\end{equation}
and the remainder terms are bounded by
\begin{equation}
 R_{h,N}(t)=\mathcal{O}(t\epsilon^{N}),\ \ \ \  S_{h,N}(t)=\mathcal{O}(t\epsilon^{N-1}).\\
\label{remainder}%
\end{equation}
The coefficient functions as well as   their derivatives up to
arbitrary order $M$ are bounded by
\begin{equation}
 \zeta^k_{h,1}=
\left\{\begin{aligned}
&\mathcal{O}(\epsilon^k)\quad \ \  \mathrm{when}\ k\  \mathrm{even},\\
&\mathcal{O}(\epsilon^{k+2})\ \ \mathrm{when}\ k\  \mathrm{odd},\\
\end{aligned}\right.\qquad \quad  \zeta^k_{h,2}=
\left\{\begin{aligned}
&\mathcal{O}(\epsilon^{k+2})\  \ \mathrm{when}\ k\  \mathrm{even},\\
&\mathcal{O}(\epsilon^{k})\quad\ \ \mathrm{when}\ k\  \mathrm{odd},\\
\end{aligned}\right.
\label{coefficient func1}%
\end{equation}
for $k=0,\ldots,N+1$  and
\begin{equation}
 \eta^k_{h,1}=
\left\{\begin{aligned}
&\mathcal{O}(\epsilon^{k-1})\quad \ \ \mathrm{when}\ k\  \mathrm{even},\\
&\mathcal{O}(\epsilon^{k+1})\quad \ \  \mathrm{when}\ k\  \mathrm{odd},\\
\end{aligned}\right.\qquad \quad  \eta^k_{h,2}=
\left\{\begin{aligned}
&\mathcal{O}(\epsilon^{k+1})\quad \mathrm{when}\ k\  \mathrm{even},\\
&\mathcal{O}(\epsilon^{k-1})\quad \mathrm{when}\ k\  \mathrm{odd},\\
\end{aligned}\right.
\label{coefficient func2}%
\end{equation} for $k=1,\ldots,N+1$,   in addition to
\begin{equation}\eta^0_{h,1}=\mathcal{O}(1),\quad \ \ \eta^0_{h,2}=\mathcal{O}(\epsilon).\label{eta21}%
\end{equation}
 Moreover, we have
$\zeta_h^{-k}=\overline{\zeta_h^{k}}$ and
$\eta_h^{-k}=\overline{\eta_h^{k}}$ for all $k$. With these bounds,
the functions $\zeta_h^k$ and $\eta_h^k$ are unique up to the terms
with order of magnitude $\mathcal{O}(\epsilon^{N+2})$. The constants
symbolized by $\mathcal{O} $   are independent of $h, \epsilon $ and
$t$ with $0\leq t \leq T$, but depend on the constants $N, T$, the
derivatives of $\omega(q_1)$ and $U$ on $K$, and the maximum order
$M$ of considered derivatives of coefficient functions.
\end{prop}
\begin{proof}We  need only to briefly highlight the main
differences in the construction of the coefficients functions since
some parts of the proof are similar to those of Theorem 4.1 in
\cite{Hairer16}. Hence, the similar derivations  of initial values
and defects will be omitted for brevity.

We will prove that there exist  two functions
\begin{equation}
\begin{aligned} &q_{h}(t)=\sum\limits_{|k|\leq N+1}q^k_{h}(t)= \sum\limits_{|k|\leq N+1}
\mathrm{e}^{\mathrm{i}k\phi(t)/\epsilon}\zeta_h^k(t),\\
 &\ p_{h}(t)=\sum\limits_{|k|\leq N+1}p^k_{h}(t)=\sum\limits_{|k|\leq N+1} \mathrm{e}^{\mathrm{i}k\phi(t)/\epsilon}\eta_h^k(t)
\end{aligned}
\label{MFE-1}%
\end{equation}
with smooth  coefficients $ \zeta_h^k$ and  $\eta_h^k$,  such that
\begin{equation*}
\begin{aligned} &q^{n}=q_{h}(t)+\mathcal{O}(\epsilon^{N+2}),\ \ p^{n}= p_{h}(t)+\mathcal{O}(\epsilon^{N+2})
\end{aligned}
\label{MFE-def}%
\end{equation*}
 for $t=nh$.
 \vskip2mm \textbf{I. Construction of the coefficients
functions.}

 $\bullet$  \textbf{1.1 Construction of the coefficients
functions for $Q^{n+\frac{1}{2}}$.}

  We consider the  following modulated Fourier expansion
\begin{equation}
\begin{aligned} &\tilde{q}_{h}(t+\frac{h}{2}):=\sum\limits_{|k|\leq N+1}
\mathrm{e}^{\mathrm{i}k \phi(t)/\epsilon}\xi_h^k(t+\frac{h}{2})
\end{aligned}
\label{MFE-2}%
\end{equation}
for $Q^{n+\frac{1}{2}}$ in the  first   term  of the ERKN integrator
\eqref{erknSchems}. Inserting  \eqref{MFE-1} and \eqref{MFE-2} into
 the first equation of  \eqref{erknSchems} and comparing the coefficients of $\mathrm{e}^{\mathrm{i}k
  \phi(t)/\epsilon}$ yields
\begin{equation*}
\begin{aligned}
&\xi^k_h(t+\frac{h}{2})=
\cos(\frac{1}{2}h\Omega)\zeta_h^k(t)+\frac{1}{2}h
\mathrm{sinc}(\frac{1}{2}h\Omega)\eta_h^k(t).
\end{aligned}
\label{MFE-3}%
\end{equation*}
Likewise, for  the modulated Fourier expansion of
$Q^{n-\frac{1}{2}}$, we obtain  the function
\begin{equation*}
\begin{aligned} &\tilde{q}_{h}(t-\frac{h}{2}):=\sum\limits_{|k|\leq N+1}
 \mathrm{e}^{\mathrm{i}k \phi(t)/\epsilon}\xi_h^k(t-\frac{h}{2})
\end{aligned}
\label{MFE-12}%
\end{equation*}
with
\begin{equation*}
\begin{aligned}
&\xi^k_h(t-\frac{h}{2})=
\cos(\frac{1}{2}h\Omega)\zeta_h^k(t)-\frac{1}{2}h
\mathrm{sinc}(\frac{1}{2}h\Omega)\eta_h^k(t).
\end{aligned}
\label{MFE-13}%
\end{equation*}

Under the conditions \eqref{coefficient func1} and
\eqref{coefficient func2}, we have
\begin{equation}
\xi^k_{h,1}(t\pm\frac{h}{2})= \left\{\begin{aligned}
&\mathcal{O}(\epsilon^k),\quad \ \ \mathrm{when}\ k\  \mathrm{even}\\
&\mathcal{O}(\epsilon^{k+2}),\ \ \mathrm{when}\ k\  \mathrm{odd}\\
\end{aligned}\right.\qquad    \xi^k_{h,2}(t\pm\frac{h}{2})=
\left\{\begin{aligned}
&\mathcal{O}(\epsilon^{k+2}),\  \ \mathrm{when}\ k\  \mathrm{even}\\
&\mathcal{O}(\epsilon^{k}),\quad\ \ \mathrm{when}\ k\  \mathrm{odd}\\
\end{aligned}\right.
\label{xi func1}%
\end{equation}
for $k=0,\ldots,N+1$, and $\xi_h^{-k}=\bar{\xi}_h^k.$

 $\bullet$   \textbf{1.2 Construction of the coefficients
functions for $q^{n}$.}

For the second term of the ERKN integrator    \eqref{erknSchems}, it
follows from its
 symmetry that
\begin{equation*}
\begin{aligned}&q_{h}(t+h)-
2\cos(h\Omega)q_{h}(t)+q_{h}(t-h)\\=& h^2\bar{b}_1(h\Omega)
\big[g(\tilde{q}_{h}(t+\frac{h}{2}))+g(\tilde{q}_{h}(t-\frac{h}{2}))\big].
\end{aligned}\label{MFE-q2-newt}%
\end{equation*}
Taking into account   the operator defined in \cite{hairer2006}
\begin{equation*}
\begin{aligned}\mathcal{L}(hD):&=e^{hD}-2\cos(h\Omega)+e^{-hD}=2\big(\cos(\mathrm{i}
hD)-\cos(h\Omega)\big)\\
&= 4\sin\big(\frac{1}{2}h\Omega+\frac{1}{2}\mathrm{i}
hD\big)\sin\big(\frac{1}{2}h\Omega-\frac{1}{2}\mathrm{i}
hD\big),\end{aligned}
\end{equation*}
the above formula can be expressed as
\begin{equation*}
\begin{aligned}&\mathcal{L}(hD)q_{h}(t)=\sum\limits_{|k|\leq
N+1}\mathcal{L}(hD)q^k_{h}(t)=
h^2\bar{b}_1(h\Omega)\big[g(\tilde{q}_{h}(t+\frac{h}{2}))+g(\tilde{q}_{h}(t-\frac{h}{2}))\big].
\end{aligned} %
\end{equation*}
 Then, from
$q^k_{h}(t+h)=\mathrm{e}^{\mathrm{i}k\phi(t+h)/\epsilon}\zeta_h^k(t+h),$
 we expand
$q^k_{h}(t+h)$   in a similar way used in \cite{Hairer16}
\begin{equation*}
\begin{aligned}&q^k_{h}(t+h)
=\big(\zeta_h^k(t)+h\dot{\zeta}_h^k(t)+\frac{h^2}{2}\ddot{\zeta}_h^k(t)+\cdots\big)
\mathrm{e}^{\mathrm{i}k\phi(t)/\epsilon}\mathrm{e}^{\mathrm{i}k\kappa\dot{\phi}(t)}
\big(1+ \mathrm{i}k\kappa\frac{h}{2}\ddot{\phi}(t)+\cdots\big),
\end{aligned}
\end{equation*}
where $\kappa=\frac{h}{\epsilon}$.  Therefore, inserting this into
the right hand size of the formula
$$\mathcal{L}(hD)q^k_{h}(t)=q^k_{h}(t+h)-2\cos(h\Omega)q^k_{h}(t)+q^k_{h}(t-h),$$ we obtain
\begin{equation*}
\begin{aligned}\mathcal{L}(hD)q^k_{h}(t)=&\Big(\mathcal{L}(\mathrm{i}k\kappa \dot{\phi}(t))\zeta_h^k(t)
+2\mathrm{i}h\sin(k\kappa\dot{\phi}(t))\dot{\zeta}_h^k(t)\\
&+\mathrm{i}k\kappa
h\cos(k\kappa\dot{\phi}(t))\ddot{\phi}(t)\zeta_h^k(t)+\cdots \Big)
\mathrm{e}^{\mathrm{i}k\phi(t)/\epsilon}
\end{aligned}
\end{equation*}
for $k\neq0.$ For $k=0,$ we have
$\mathcal{L}(hD)q^0_{h}(t)=\mathcal{L}(hD)\zeta^0_{h}(t).$ Comparing
the coefficients of $\mathrm{e}^{\mathrm{i}k\phi(t)/\epsilon}$
yields the  following formal equations:
\begin{equation}\label{zeta0 exp}
\begin{aligned}\epsilon^{-2}\mathcal{L}(hD)\zeta_h^0(t)
=&\kappa^2\bar{b}_1(h\Omega)\Big[g(\xi^0_h(t+\frac{h}{2}))+
\sum\limits_{s(\alpha)=0}\frac{1}{m!}g^{(m)}(\xi^0_h(t+\frac{h}{2}))(\xi_h(t+\frac{h}{2}))^{\alpha} \\
& + g(\xi^0_h(t-\frac{h}{2}))+
\sum\limits_{s(\alpha)=0}\frac{1}{m!}g^{(m)}(\xi^0_h(t-\frac{h}{2}))(\xi_h(t-\frac{h}{2}))^{\alpha}\Big]
\end{aligned} %
\end{equation}
and
\begin{equation}\label{zetak exp}
\begin{aligned}
&\epsilon^{-2}\Big(\mathcal{L}(\mathrm{i}k\kappa
\dot{\phi}(t))\zeta_h^k(t)
+2\mathrm{i}h\sin(k\kappa\dot{\phi}(t))\dot{\zeta}_h^k(t)+\mathrm{i}k\kappa
h\cos(k\kappa\dot{\phi}(t))\ddot{\phi}(t)\zeta_h^k(t)+\cdots\Big)\\
=&\kappa^2\bar{b}_1(h\Omega)\Big(
\sum\limits_{s(\alpha)=k}\frac{1}{m!}g^{(m)}(\xi^0_h(t+\frac{h}{2}))(\xi_h(t+\frac{h}{2}))^{\alpha}
+
\sum\limits_{s(\alpha)=k}\frac{1}{m!}g^{(m)}(\xi^0_h(t-\frac{h}{2}))(\xi_h(t-\frac{h}{2}))^{\alpha}\Big)
\end{aligned} %
\end{equation}
for $k\neq0.$ Here the sum ranges over
$\alpha=(\alpha_1,\ldots,\alpha_m)$ with   $0<|\alpha_i|\leq N+1$,
$s(\alpha)=\sum\limits_{j=1}^{m}\alpha_j,$ and $(\xi_h(t))^{\alpha}$
is an abbreviation for
$(\xi^{\alpha_1}_h(t),\ldots,\xi^{\alpha_m}_h(t))$.

\textit{\textbf{Proof of \eqref{phase func}.}} Under the conditions
\eqref{coefficient func1}, \eqref{coefficient func2} and \eqref{xi
func1} and with careful observations, we find the dominating terms
on both sides of \eqref{zetak exp} appearing for $|k|=1$ and for
$\zeta_{h,2}^{\pm1}(t)$. The corresponding relation is
\begin{equation}\label{fir phi}
\begin{aligned}
\mathcal{L}(\pm\mathrm{i}\kappa
\dot{\phi}(t))\zeta_{h,2}^{\pm1}(t)=2\kappa^2\bar{b}_1(h\upsilon)\big(\omega_0^2-\omega^2(\zeta_{h,1}^{0}(t))\big)\zeta_{h,2}^{\pm1}(t).
\end{aligned} %
\end{equation}
This means that
\begin{equation}\label{rea phi}
\begin{aligned}
 \cos(\mp  \kappa \dot{\phi}(t))
%4\sin\big(\frac{1}{2}h\upsilon\mp\frac{1}{2} \kappa
%\dot{\phi}\big)\sin\big(\frac{1}{2}h\upsilon\pm\frac{1}{2} \kappa
%\dot{\phi}\big)
=\cos(h\upsilon)+ \bar{b}_1(h\upsilon)
h^2\upsilon^2-\kappa^2\bar{b}_1(h\upsilon)
\omega^2(\zeta_{h,1}^{0}(t)),
\end{aligned} %
\end{equation}
 in which we have utilised
$\xi_{h,1}^{0}(t)=\zeta_{h,1}^{0}(t)+\mathcal{O}(h).$
%If $\bar{b}_1(h\upsilon)$ is chosen as
%$$\bar{b}_1(h\upsilon)=\frac{1}{2}
%\mathrm{sinc}^2(h\upsilon/2)=\frac{1-\cos(h\upsilon)}{h^2\upsilon^2},$$
With the special choice of $\bar{b}_1(h\upsilon)$ given in this
theorem, \eqref{rea phi} becomes
\begin{equation}\label{rea phi-1}
\begin{aligned}
 1-\cos( \kappa \dot{\phi}(t))
%4\sin\big(\frac{1}{2}h\upsilon\mp\frac{1}{2} \kappa
%\dot{\phi}\big)\sin\big(\frac{1}{2}h\upsilon\pm\frac{1}{2} \kappa
%\dot{\phi}\big)
= 2\sin^2(\kappa \dot{\phi}(t)/2)=\kappa^2\frac{1}{2}
\mathrm{sinc}^2(h\upsilon/2) \omega^2(\zeta_{h,1}^{0}(t)),
\end{aligned} %
\end{equation}
which proves \eqref{phase func}.

\textit{\textbf{Relation to $ \zeta^0_{h}(t)$.}} According to
\eqref{zeta0 exp}, we obtain  a relation to $ \zeta^0_{h}(t)$ as
follows:
\begin{equation*}
\begin{aligned}  \ddot{\zeta}_{h,1}^0(t)
=&\bar{b}_1(0)\Big[-\frac{2\omega(\xi^0_{h,1}(t+\frac{h}{2}))
|\xi^1_{h,2}(t+\frac{h}{2})|^2}{\epsilon^2}\nabla_{q_1}\omega(\xi^0_{h,1}(t+\frac{h}{2}))
\\
&-\frac{2\omega(\xi^0_{h,1}(t-\frac{h}{2}))
|\xi^1_{h,2}(t-\frac{h}{2})|^2}{\epsilon^2}\nabla_{q_1}\omega(\xi^0_{h,1}(t-\frac{h}{2}))\\
&-\nabla_{q_1}U(\xi^0_{h}(t+\frac{h}{2}))-\nabla_{q_1}U(\xi^0_{h}(t-\frac{h}{2}))
+\mathcal{O}(\epsilon)\Big],\\
\zeta^0_{h,2}(t)=&\frac{\epsilon^{2}\kappa^2\bar{b}_1(h\nu)}{4\sin^2(h\nu/2)}\Big[-\frac{\omega^2(\xi^0_{h,1}(t+\frac{h}{2}))
-\omega_0^2}{\epsilon^2}\xi^0_{h,2}(t+\frac{h}{2})-\nabla_{q_2}U(\xi^0_{h}(t+\frac{h}{2}))
\\
&-\frac{\omega^2(\xi^0_{h,1}(t-\frac{h}{2}))
-\omega_0^2}{\epsilon^2}\xi^0_{h,2}(t-\frac{h}{2})-\nabla_{q_2}U(\xi^0_{h}(t-\frac{h}{2}))
+\mathcal{O}(\epsilon)\Big].
\end{aligned} %
\end{equation*}

\textit{\textbf{Relation to $ \zeta^{\pm1}_{h}(t)$.}} It follows
from \eqref{zetak exp} that
\begin{equation*}
\begin{aligned}  4 \sin^2(\kappa
 \dot{\phi}(t)/2) \zeta_{h,1}^{\pm1}(t)
=&h^2\bar{b}_1 (0)\Big(
\sum\limits_{s(\alpha)=\pm1}\frac{1}{m!}g^{(m)}(\xi^0_h(t+\frac{h}{2}))(\xi_h(t+\frac{h}{2}))^{\alpha}
\\&+
\sum\limits_{s(\alpha)=\pm1}\frac{1}{m!}g^{(m)}(\xi^0_h(t-\frac{h}{2}))(\xi_h(t-\frac{h}{2}))^{\alpha}\Big),
\end{aligned} %
\end{equation*}
which gives the relation to  $\zeta_{h,1}^{\pm1}(t)$. For
$\zeta_{h,2}^{\pm1},$ we have discussed the $\epsilon^{-1}$-terms in
\eqref{zetak exp}. For the $\epsilon^{0}$-terms, we obtain
 the following relation
\begin{equation}\label{zeteh2}
\begin{aligned}   2\mathrm{i}h\sin(\kappa\dot{\phi}(t))\dot{\zeta}_{h,2}^{\pm1}(t)+\mathrm{i}\kappa
h\cos(\kappa\dot{\phi}(t))\ddot{\phi}(t)\zeta_{h,2}^{\pm1}(t)
=&h^2\bar{b}_1(h\upsilon) \mathcal{O}(\epsilon).
\end{aligned} %
\end{equation}
%which yields
%\begin{equation*}
%\begin{aligned} \frac{ \dot{\zeta}_{h,2}^{\pm1}(t)}{\epsilon}=-\frac{1}{2}\frac{ \cos(\kappa\dot{\phi}(t))}{\sin(\kappa\dot{\phi}(t))}
%\ddot{\phi}(t)+ \mathcal{O}(\epsilon).\\
%\end{aligned} %
%\end{equation*}
With the formula \eqref{phase func}, it can be deduced that
\begin{equation*}\begin{aligned}
&\sin(\kappa\dot{\phi}(t))=2\sin(\kappa\dot{\phi}(t)/2)\cos(\kappa\dot{\phi}(t)/2)\\
&\qquad\qquad\ =\kappa \mathrm{sinc}(h\upsilon/2)
\omega(\zeta_{h,1}^{0}(t))\sqrt{1-\frac{\kappa^2}{4}
\mathrm{sinc}^2(h\upsilon/2) \omega^2(\zeta_{h,1}^{0}(t))},\\
&\cos(\kappa\dot{\phi}(t))=1-2\sin^2(\kappa\dot{\phi}(t)/2)=1-\frac{\kappa^2}{2}
\mathrm{sinc}^2(h\upsilon/2) \omega^2(\zeta_{h,1}^{0}(t)),\\
&\ddot{\phi}(t)=(\frac{2}{\kappa}\arcsin\big(\frac{\kappa}{2}
\mathrm{sinc}(h\upsilon/2)
\omega(\zeta_{h,1}^{0}(t))\big)'=\frac{2}{\kappa}\frac{\frac{\kappa}{2}
\mathrm{sinc}(h\upsilon/2) \nabla
_{q_1}\omega(\zeta_{h,1}^{0}(t))\dot{\zeta}_{h,1}^{0}(t)}{\sqrt{1-\frac{\kappa^2}{4}
\mathrm{sinc}^2(h\upsilon/2) \omega^2(\zeta_{h,1}^{0}(t))}}.
\end{aligned} \end{equation*}
Inserting all of these expressions  into \eqref{zeteh2} leads to
\begin{equation*}\label{zeteh21}
\begin{aligned}    \frac{\dot{\zeta}_{h,2}^{\pm1}(t)}{\epsilon}=-\frac{\nabla
_{q_1}\omega(\zeta_{h,1}^{0}(t))\dot{\zeta}_{h,1}^{0}(t)}{2
\omega(\zeta_{h,1}^{0}(t))}
\frac{\zeta_{h,2}^{\pm1}(t)}{\epsilon}\Big(\frac{1-\frac{\kappa^2}{2}
\mathrm{sinc}^2(h\upsilon/2)
\omega^2(\zeta_{h,1}^{0}(t))}{1-\frac{\kappa^2}{4}
\mathrm{sinc}^2(h\upsilon/2) \omega^2(\zeta_{h,1}^{0}(t))}\Big)+
\mathcal{O}(\epsilon).
\end{aligned} %
\end{equation*}

%\begin{equation}
% \sin(\frac{h\dot{\phi}(t)}{2\epsilon}
% )=\frac{h}{2\epsilon}
%\mathrm{sinc}(h\upsilon/2) \omega(\zeta_{h,1}^{0}(t)),\ \ \
%\phi(0)=0
%\end{equation}

\textit{\textbf{Relation to $ \zeta^{ k}_{h}(t)$.}} The relation to
$ \zeta^{k}_{h,1}(t)$ for $k>1$ is
\begin{equation*}
\begin{aligned}  4 \sin^2(k \kappa
 \dot{\phi}/2) \zeta_{h,1}^{k}(t)
=&h^2\bar{b}_1(0) \Big(
\sum\limits_{s(\alpha)=k}\frac{1}{m!}g_1^{(m)}(\xi^0_h(t+\frac{h}{2}))(\xi_h(t+\frac{h}{2}))^{\alpha}
\\&+
\sum\limits_{s(\alpha)=k}\frac{1}{m!}g_1^{(m)}(\xi^0_h(t-\frac{h}{2}))(\xi_h(t-\frac{h}{2}))^{\alpha}\Big).\\
\end{aligned} %
\end{equation*}
The negative coefficient of  $ \zeta^{k}_{h,2}(t)$ for $k>1$ in
\eqref{zetak exp} equals
\begin{equation*}
\begin{aligned}&\frac{2( \cos(k\kappa \dot{\phi})
-\cos(h\upsilon))}{h^2}-2\bar{b}_1(h\upsilon) \frac{
\omega^2(\zeta_{h,1}^{0}(t))-\omega_0^2}{\epsilon^2}\\
=&\frac{2( \cos(k\kappa \dot{\phi}) -\cos(h\upsilon))}{h^2}-
\frac{2( \cos(\kappa \dot{\phi}) -\cos(h\upsilon))}{h^2}\\
=&\frac{2}{h^2}( \cos(k\kappa \dot{\phi}) -\cos(\kappa \dot{\phi}))
=-\frac{4}{h^2} \sin(\frac{k-1}{2}\kappa \dot{\phi})\sin(\frac{k+1}{2}\kappa \dot{\phi}).\\
\end{aligned} %
\end{equation*}
Thus, the relation  to  $ \zeta^{k}_{h,2}(t)$ for $k>1$ is
\begin{equation*}
\begin{aligned} & \frac{4}{h^2}  \sin(\frac{k-1}{2}\kappa \dot{\phi})\sin(\frac{k+1}{2}\kappa
\dot{\phi})
\zeta_{h,2}^{k}(t)\\
=&h^2\bar{b}_1(h\upsilon) \Big(
\sum\limits_{s(\alpha)=k}\frac{1}{m!}g_2^{(m)}(\xi^0_h(t+\frac{h}{2}))(\xi_h(t+\frac{h}{2}))^{\alpha}\\
&+
\sum\limits_{s(\alpha)=k}\frac{1}{m!}g_2^{(m)}(\xi^0_h(t-\frac{h}{2}))(\xi_h(t-\frac{h}{2}))^{\alpha}\Big).\\
\end{aligned} %
\end{equation*}
It follows from \eqref{numerical cond} and \eqref{phase func} that
\begin{equation*}
\begin{aligned}& \sin(\kappa\dot{\phi}(t)/2)=\frac{\kappa}{2}
\mathrm{sinc}(h\upsilon/2)
\omega(\zeta_{h,1}^{0}(t))\\
&=\frac{\kappa}{2} \mathrm{sinc}(h\upsilon/2)
\omega(q_{1}^{n})+\mathcal{O}(\epsilon)\leq\sin(\frac{\pi}{N+2}),\end{aligned} %
\end{equation*}
which leads to $\kappa\dot{\phi}(t)/2\leq
\frac{\pi}{N+2}+\mathcal{O}(\epsilon).$  This implies that the
factor of $\sin^2(k \kappa \dot{\phi}/2)$ is bounded away from $0$
for $|k|\leq N+1$, and the factor of $\sin(\frac{k-1}{2}\kappa
\dot{\phi})\sin(\frac{k+1}{2}\kappa \dot{\phi})$ is bounded away
from $0$ for $|k|\leq N.$

 $\bullet$   \textbf{1.3 Construction of the coefficients functions
for $p^{n}$.}

For the third term of   \eqref{erknSchems}, we get  a similar
relation
\begin{equation*}
\begin{aligned}&p_{h,2}(t+h)-2\cos(h\omega)p_{h,2}(t)+p_{h,2}(t-h)\\=&hb_1(h\omega)\big[g_2(\tilde{q}_{h}(t+\frac{h}{2}))-g_2(\tilde{q}_{h}(t-\frac{h}{2})\big].
\end{aligned}\label{MFE-q2-newtt}%
\end{equation*}
 As we have derived for   the second term,
we can  obtain the relations for  the modulated Fourier function of
$p^{n}$.   Moreover, by considering the second and third formulae of
\eqref{erknSchems} and
 the choice of $\bar{b}_1$ and $b_1$, we obtain
 \begin{equation*}
 q^{n+1}-q^{n} =\upsilon^{-1}\tan(h\upsilon/2) ( p^{n+1}+ p^n ),
\end{equation*}
which can be represented as
 \begin{equation}\label{rea pq}
 \mathcal{L}_1(hD) q^k_{h}(t) =\upsilon^{-1}\tan(h\upsilon/2) p^k_{h}(t)
\end{equation}
by defining
\begin{equation*}
\mathcal{L}_1(hD)=(\mathrm{e}^{hD}-1)(\mathrm{e}^{hD}+1)^{-1}.
\end{equation*}
Based on   this relationship between $q^k_{h}$ and $p^k_{h}$,  the
connection between $\zeta^k_{h}$ and $\eta^k_{h}$ can be obtained.
As an example,
 when $k=0$, one has
\begin{equation}\label{initial eta modula sys}
\eta^0_{h}(t)=\dot{\zeta}^0_{h}(t)+\mathcal{O}(h^2).
\end{equation}

% As we have derived for   the second term,
%we can  obtain the relations for  the modulated Fourier function of
%$p^{n}$.

The above relations   allow  us to construct the functions and
arrive at the bounds \eqref{coefficient func1} and
\eqref{coefficient func2} in the same way used in the proofs of
Theorems 2.1 and 4.1 in \cite{Hairer16}.

 \textbf{II. Initial values.}

 The initial
values for the differential equations of $\zeta_{h,1}^0,
\zeta_{h,2}^{\pm1}$ and $\eta_{h,1}^0, \eta_{h,2}^{\pm1}$ are
  computed from $q^0$ and $p^0$, which are
determined from the conditions that  \eqref{MFE-ERKN} is satisfied
without the remainder term for $t = 0$ and $t = h$.  We note that
the initial value $\dot{\zeta}_{h,1}^0(0)$ is obtained by
considering \eqref{initial eta modula sys} and
$p^0_{1}=\eta^0_{h,1}(0)+\mathcal{O}(\epsilon)=\dot{\zeta}^0_{h,1}(0)+\mathcal{O}(\epsilon).$

 \textbf{III. Defect.}

 The defect \eqref{remainder} can be obtained   by the same arguments
as in the constant-frequency case researched  in
\cite{Hairer00,hairer2006,17-new}.

 The proof is complete. \hfill    \end{proof}

\section{Action conservation of the integrator (proof of Theorem \ref{pre mod I})} \label{sec:Almost-invariants}
This section will prove the action conservation of the ERKN
integrator.

Similarly to \cite{Hairer16}, the modulated Fourier expansion is
considered  on an interval of length $O(h)$ and the coefficient
functions   and the phase function $\phi(t)$ are replaced by Taylor
polynomials of degree $M\geq N+3$.
 Let
$$\zeta=\big(\zeta^{-N+1}_h(t),\cdots,\zeta^{-1}_h(t),
\zeta^{0}_h(t),\zeta^{1}_h(t),\cdots,\zeta^{N-1}_h(t)\big),$$ and
$$\eta=\big(\eta^{-N+1}_h(t),\cdots,\eta^{-1}_h(t),
\eta^{0}_h(t),\eta^{1}_h(t),\cdots,\eta^{N-1}_h(t)\big). $$ We
obtain the following result about an invariant of
  the modulated Fourier expansion and its relationship with the
modified actions.

\begin{prop}\label{second invariant thm}
There exists a function $\mathcal{I}[\zeta,\eta]$ such that
\begin{equation*}
\begin{aligned}
&\mathcal{I}[\zeta,\eta](t)=\mathcal{I}[\zeta,\eta](0)+\mathcal{O}(t\epsilon^{N}) \ \ \textmd{for}\ 0\leq t\leq h,\\
&\mathcal{I}[\zeta,\eta](nh)=I_h(q_n,p_n)+\mathcal{O}(\epsilon)  \ \
\textmd{for}\ n=0,1,
\label{II}%
\end{aligned}
\end{equation*}
  where the constants  symbolized by
$\mathcal{O}$ are independent of $n, h$ and $\epsilon$. The notation
$\mathcal{I}[\zeta,\eta]$  depends on $\zeta, \dot{\zeta},\ldots,
\zeta^{(M)}$  and
 $\eta,  \dot{\eta},\ldots, \eta^{(M)}$.
\end{prop}
\begin{proof} Following \cite{Hairer13} and without loss of
generality, the coefficient functions $\zeta^k_h(t), \eta^k_h(t)$
and the phase function $\phi(t)$ can be assumed to be   polynomials
of degree at most $M$. Therefore, $q^k_{h}(t)$ and  $p^k_{h}(t)$ are
entire analytic functions of $t$.

It then from the analysis presented in the proof of Proposition
\ref{energy thm} that
\begin{equation*}
\begin{aligned} &\tilde{q}_{h}(t+\frac{h}{2})=\cos(\frac{1}{2}h\Omega)q_{h}(t)+\frac{1}{2}h \mathrm{sinc}(\frac{1}{2}h\Omega)p_h(t),\\
&\tilde{q}_{h}(t-\frac{h}{2})=\cos(\frac{1}{2}h\Omega)q_{h}(t)-\frac{1}{2}h \mathrm{sinc}(\frac{1}{2}h\Omega)p_h(t),\\
& \mathcal{L}(hD) q_{h}(t)=h^2\bar{b}_1(h\Omega)\big(g(\tilde{q}_{h}(t+\frac{h}{2}))+g(\tilde{q}_{h}(t-\frac{h}{2}))\big)+\mathcal{O}(\epsilon^{N}),\\
& \mathcal{L}(hD)
p_{h}(t)=hb_1(h\Omega)\big(g(\tilde{q}_{h}(t+\frac{h}{2}))-g(\tilde{q}_{h}(t-\frac{h}{2}))\big)+\mathcal{O}(\epsilon^{N}),
\end{aligned}
\label{methods-inva}%
\end{equation*}
where  we have used the following denotations
\begin{equation*}
\begin{aligned}q_{h}(t)=\sum\limits_{ |k|\leq N+1}q^k_{h}(t),\  \ p_{h}(t)=\sum\limits_{
|k|\leq N+1}p^k_{h}(t),\ \
\tilde{q}_{h}(t\pm\frac{1}{2}h)=\sum\limits_{ |k|\leq
N+1}\tilde{q}^k_{h}(t\pm\frac{1}{2}h)
\end{aligned}
\end{equation*}
with
\begin{equation*}\label{rev-smo fun}
\begin{aligned}q^k_{h}(t)=\mathrm{e}^{\mathrm{i}k \phi(t)/\epsilon}\zeta_h^k(t),\ \ p^k_{h}(t)= \mathrm{e}^{\mathrm{i}k \phi(t)/\epsilon}\eta_h^k(t),\ \
\tilde{q}^k_{h}(t\pm\frac{1}{2}h)=\mathrm{e}^{\mathrm{i}k
\phi(t)/\epsilon}\xi_h^k(t\pm\frac{1}{2}h)
\end{aligned}
\end{equation*}
and
\begin{equation*}
\begin{aligned} \xi^k_h(t\pm\frac{h}{2})=
\cos(\frac{1}{2}h\Omega)\zeta_h^k(t)\pm\frac{1}{2}h
\mathrm{sinc}(\frac{1}{2}h\Omega)\eta_h^k(t) .
\end{aligned}
\end{equation*}
Rewriting the above equations in terms of $\tilde{q}^k_h,\ q_h^k,\
p_h^k$ yields
\begin{equation}
\begin{aligned} \tilde{q}^k_{h}(t+\frac{h}{2})=&\cos(\frac{1}{2}h\Omega)q^k_{h}(t)+\frac{1}{2}h \mathrm{sinc}(\frac{1}{2}h\Omega)p^k_h(t),\\
\tilde{q}^k_{h}(t-\frac{h}{2})=&\cos(\frac{1}{2}h\Omega)q^k_{h}(t)-\frac{1}{2}h \mathrm{sinc}(\frac{1}{2}h\Omega)p^k_h(t),\\
 \mathcal{L}(hD)
q^k_{h}(t)=&-h^2\bar{b}_1(h\Omega)\Big(\frac{1}{\epsilon^2}\nabla_{q^{-k}}\mathcal{V}(\tilde{q}
(t+\frac{h}{2}))+\nabla_{q^{-k}}\mathcal{U}(\tilde{q}
(t+\frac{h}{2}))
\\
&+\frac{1}{\epsilon^2}\nabla_{q^{-k}}\mathcal{V}(\tilde{q}
(t-\frac{h}{2}))+\nabla_{q^{-k}}\mathcal{U}(\tilde{q} (t-\frac{h}{2}))\Big)+\mathcal{O}(\epsilon^{N}),\\
 \mathcal{L}(hD)
p^k_{h}(t)=&-hb_1(h\Omega)\Big(\frac{1}{\epsilon^2}\nabla_{q^{-k}}\mathcal{V}(\tilde{q}
(t+\frac{h}{2}))+\nabla_{q^{-k}}\mathcal{U}(\tilde{q}
(t+\frac{h}{2}))
\\
&-\frac{1}{\epsilon^2}\nabla_{q^{-k}}\mathcal{V}(\tilde{q}
(t-\frac{h}{2}))-\nabla_{q^{-k}}\mathcal{U}(\tilde{q}
(t-\frac{h}{2}))\Big)+\mathcal{O}(\epsilon^{N}),
\end{aligned}
\label{methods-inva-nnew}%
\end{equation}
where $\mathcal{U}(\tilde{q})$ and $\mathcal{V}(\tilde{q})$ are
determined by
\begin{equation*}
\begin{aligned}
&\mathcal{U}(\tilde{q}(t\pm\frac{h}{2}))=U(\tilde{q}^0_h(t\pm\frac{h}{2}))+
\sum\limits_{m=2}^{N+1}
\sum\limits_{s(\alpha)=0}\frac{1}{m!}U^{(m)}(\tilde{q}^0_h(t\pm\frac{h}{2}))
(\tilde{q}_h(t\pm\frac{h}{2}))^{\alpha},\\
&\mathcal{V}(\tilde{q}(t\pm\frac{h}{2}))=V(\tilde{q}^0_h(t\pm\frac{h}{2}))+
\sum\limits_{m=2}^{N+1}
\sum\limits_{s(\alpha)=0}\frac{1}{m!}V^{(m)}(\tilde{q}^0_h(t\pm\frac{h}{2}))
(\tilde{q}_h(t\pm\frac{h}{2}))^{\alpha}
\end{aligned}
\label{newuu}%
\end{equation*}
with
\begin{equation*}
\begin{aligned}
&V(q)=\frac{\omega(q_1)^2-\omega_0^2}{2}
|q_2|^2,\\
&\tilde{q}(t\pm\frac{h}{2})=\big(\tilde{q}^{-N+1}_h(t\pm\frac{h}{2}),\ldots,
\tilde{q}^{0}_h(t\pm\frac{h}{2}),\ldots,\tilde{q}^{N-1}_h(t\pm\frac{h}{2})\big).
\end{aligned}
\end{equation*}

\textit{\textbf{Proof of the first statement.}}
 Consider the vector function $\tilde{q}(\lambda,t\pm\frac{h}{2})$ of $\lambda$
below
$$\tilde{q}(\lambda,t\pm\frac{h}{2})=\big( \mathrm{e}^{\mathrm{i}(-N+1)\lambda
/\epsilon}\tilde{q}^{-N+1}_h(t\pm\frac{h}{2}),\cdots,
\tilde{q}^{0}_h(t\pm\frac{h}{2}),\cdots,\mathrm{e}^{\mathrm{i}(N-1)\lambda
/\epsilon}\tilde{q}^{N-1}_h(t\pm\frac{h}{2})\big).$$  It is clear
  that $\mathcal{V}( \tilde{q}(\lambda,t\pm\frac{h}{2}))$ and $\mathcal{U}( \tilde{q}(\lambda,t\pm\frac{h}{2}))$ do  not depend on
$\lambda$. Thus, we obtain
\begin{equation*}
\begin{aligned}0=&\frac{1}{2}\frac{d}{d\lambda}\Big[\frac{ \mathcal{V}( \tilde{q}(\lambda,t+\frac{h}{2}))}{\epsilon^2}+
 \mathcal{U}( \tilde{q}(\lambda,t+\frac{h}{2}))+\frac{\mathcal{V}(
 \tilde{q}(\lambda,t-\frac{h}{2}))}{\epsilon^2} +
 \mathcal{U}(
\tilde{q}(\lambda,t-\frac{h}{2}))\Big]\\
=&\frac{1}{2}\sum\limits_{|k|\leq
N+1}\frac{\mathrm{i}k}{\epsilon}\mathrm{e}^{\mathrm{i}k\lambda
/\epsilon} (\tilde{q}^{k}_h(t+\frac{h}{2}))^\intercal \Big(
\frac{\nabla_{q^{k}} \mathcal{V}(
\tilde{q}(\lambda,t+\frac{h}{2}))}{\epsilon^2}
+\nabla_{q^{k}}\mathcal{U}(
\tilde{q}(\lambda,t+\frac{h}{2}))\Big)\\
&+\frac{1}{2}\sum\limits_{|k|\leq
N+1}\frac{\mathrm{i}k}{\epsilon}\mathrm{e}^{\mathrm{i}k\lambda
/\epsilon} (\tilde{q}^{k}_h(t-\frac{h}{2}))^\intercal \Big(
\frac{\nabla_{q^{k}} \mathcal{V}(
\tilde{q}(\lambda,t-\frac{h}{2}))}{\epsilon^2}
+\nabla_{q^{k}}\mathcal{U}(
\tilde{q}(\lambda,t-\frac{h}{2}))\Big).\end{aligned}
\end{equation*}
 By letting $\lambda=0$, one has
\begin{equation}\begin{aligned}
&0=\frac{1}{2}\sum\limits_{|k|\leq N+1}\frac{\mathrm{i}k}{\epsilon}
(\tilde{q}^{k}_h(t+\frac{h}{2}))^\intercal \Big(
\frac{\nabla_{q^{k}} \mathcal{V}( \tilde{q}(\lambda,t+\frac{h}{2}))
}{\epsilon^2} +\nabla_{q^{k}}\mathcal{U}(
\tilde{q}(\lambda,t+\frac{h}{2}))\Big)\\
&+\frac{1}{2}\sum\limits_{|k|\leq N+1}\frac{\mathrm{i}k}{\epsilon}
(\tilde{q}^{k}_h(t-\frac{h}{2}))^\intercal \Big(
\frac{\nabla_{q^{k}} \mathcal{V}(
\tilde{q}(\lambda,t-\frac{h}{2}))}{\epsilon^2}
+\nabla_{q^{k}}\mathcal{U}( \tilde{q}(\lambda,t-\frac{h}{2}))\Big).
\label{sec I}%
\end{aligned}
\end{equation}
It then from the last two identities of \eqref{methods-inva-nnew}
that
\begin{equation*}
\begin{aligned}
&\frac{\nabla_{q^{k}} \mathcal{V}( \tilde{q}(\lambda,t+\frac{h}{2}))
}{\epsilon^2} +\nabla_{q^{k}}\mathcal{U}(
\tilde{q}(\lambda,t+\frac{h}{2}))\\
=&\frac{1}{2}\Big[ (-h^2\bar{b}_1(h\Omega))^{-1}\mathcal{L}(hD)
q^k_{h}(t)+(-hb_1(h\Omega))^{-1}\mathcal{L}(hD)
p^k_{h}(t)\Big]+\mathcal{O}(\epsilon^{N}),\\
&\frac{\nabla_{q^{k}} \mathcal{V}(
\tilde{q}(\lambda,t-\frac{h}{2}))}{\epsilon^2}
+\nabla_{q^{k}}\mathcal{U}(
\tilde{q}(\lambda,t-\frac{h}{2}))\\
=&\frac{1}{2}\Big[ (-h^2\bar{b}_1(h\Omega))^{-1}\mathcal{L}(hD)
q^k_{h}(t)-(-hb_1(h\Omega))^{-1}\mathcal{L}(hD)
p^k_{h}(t)\Big]+\mathcal{O}(\epsilon^{N}).
\end{aligned}
\end{equation*}
Inserting this result into \eqref{sec I} and considering the first
two identities of \eqref{methods-inva-nnew}, one obtains
\begin{eqnarray}
 0 &=& \frac{\mathrm{i}}{2\epsilon} \sum\limits_{|k|\leq N+1}k
\big(q^{-k}_{h}(t)\big)^\intercal
\cos(\frac{1}{2}h\Omega)(-h^2\bar{b}_1(h\Omega))^{-1}\mathcal{L}(hD)
q^k_{h}(t) \label{1sta er 1}\\
    &\  & + \frac{\mathrm{i}}{2\epsilon} \sum\limits_{|k|\leq N+1}k \big(p^{-k}_h(t)\big)^\intercal \frac{1}{2}h
\mathrm{sinc}(\frac{1}{2}h\Omega)(-h
b_1(h\Omega))^{-1}\mathcal{L}(hD)
p^k_{h}(t)  \label{1sta er 2}\\
    &\ & +\mathcal{O}(\epsilon^{N}). \nonumber
\end{eqnarray}
In terms of the Taylor expansion of $\mathcal{L}(hD)$ on p. 500 of
\cite{hairer2006}, we have
$$\mathcal{L}(hD)y=4\sin^2(\frac{1}{2}h\Omega ) y+\sum\limits_{l\geq0}\frac{2h^{2l+2}}{(2l+2)!}
                                                y^{(2l+2)}.
$$
With this and  the ``magic formulas" on p. 508 of \cite{hairer2006},
we deduce that (omitting the superscript $k$ and $(t)$ on
$q^{k}(t)$)
\begin{equation}\label{ld1}
\begin{aligned}
& \mathrm{Re}(\textmd{right-hand\ side\ term\ of\  \eqref{1sta er 1}}) \\
%=&\frac{\mathrm{i}}{2\epsilon} \sum\limits_{|k|\leq
%N+1}k\mathrm{Im}\Big[\big(\bar{q}_{h} \big)^\intercal
%\cos(\frac{1}{2}h\Omega)(-h^2\bar{b}_1(h\Omega))^{-1}
% 4\sin^2(\frac{1}{2}h\Omega)q_{h}  \\
% &+ \big(\bar{q}_{h} \big)^\intercal
%\cos(\frac{1}{2}h\Omega)(-h^2\bar{b}_1(h\Omega))^{-1}
%  \sum\limits_{l\geq0}\frac{2h^{2l+2}}{(2l+2)!}
% q^{(2l+2)}_{h} \Big)\Big]\\
=&\frac{\mathrm{i}}{2\epsilon} \sum\limits_{|k|\leq
N+1}k\mathrm{Im}\Big[ \big(\bar{q}_{h} \big)^\intercal
\cos(\frac{1}{2}h\Omega)(-h^2\bar{b}_1(h\Omega))^{-1}
  4\sin^2(\frac{1}{2}h\Omega )q_{h}  \Big]\\
  &-\frac{\mathrm{i}}{2\epsilon} \sum\limits_{|k|\leq N+1}k\frac{d}{dt}\sum\limits_{l\geq0}\frac{2h^{2l}}{(2l+2)!}\mathrm{Im}\Big[\big(\bar{q}_{h} \big)^\intercal
\cos(\frac{1}{2}h\Omega)( \bar{b}_1(h\Omega))^{-1}
                                                q^{(2l+1)}_{h}
                                               -\big(\bar{\dot{q}}_{h} \big)^\intercal
\\ &\cos(\frac{1}{2}h\Omega)( \bar{b}_1(h\Omega))^{-1}
                                                q^{(2l)}_{h}
+\cdots\pm  \big( \bar{q}^{(l)}_{h} \big)^\intercal
\cos(\frac{1}{2}h\Omega)( \bar{b}_1(h\Omega))^{-1}
                                                q^{(l+1)}_{h} \Big].
\end{aligned}
\end{equation}
The first expression of this result becomes zero and the last two
lines are a total derivative. Likewise, the same conclusion holds
for the term   \eqref{1sta er 2}.
 Therefore,     there exists a
function $\mathcal{I}[\zeta,\eta]$ such that
$\frac{d}{dt}\mathcal{I}[\zeta,\eta](t)=\mathcal{O}(\epsilon^{N})$
and an integration yields the first statement   of the theorem.

\textit{\textbf{Proof of the second statement.}}   It can be
obtained from the scheme of ERKN integrators that
$$2h\mathrm{sinc}(h\Omega)p_n=q_h(t+h)-q_h(t-h)+\mathcal{O}(h^3).$$
Following the analysis on p. 133 of \cite{Hairer16} and by the
Taylor series around $h=0$, we have
\begin{equation*}
\begin{aligned}p_{n,2}=&\frac{\mathrm{i} }{h}
\big(\mathrm{e}^{\mathrm{i}\phi(t)/\epsilon} \zeta_{h,2}^{1}(t)
-\mathrm{e}^{-\mathrm{i}\phi(t)/\epsilon} \zeta_{h,2}^{-1}(t)\big)
\frac{\sin\big(\frac{h}{\epsilon}\dot{\phi}(t)\big)}{\mathrm{sinc}(h\upsilon) }+\mathcal{O}(\epsilon)\\
=&\frac{\mathrm{i} }{\epsilon}
\big(\mathrm{e}^{\mathrm{i}\phi(t)/\epsilon} \zeta_{h,2}^{1}(t)
-\mathrm{e}^{-\mathrm{i}\phi(t)/\epsilon}
\zeta_{h,2}^{-1}(t)\big)\frac{\mathrm{sinc}(h\upsilon/2)}{\mathrm{sinc}(h\upsilon)
} \omega_h(\zeta_{h,1}^{0}(t))+\mathcal{O}(\epsilon ).\end{aligned}
\end{equation*}
This leads to
\begin{equation}\label{re ne1}
\begin{aligned}&\eta^1_{n,2}(t)
=\frac{\mathrm{sinc}(h\upsilon/2)}{\mathrm{sinc}(h\upsilon)}\frac{\omega_h(\zeta_{h,1}^{0}(t))}{
\epsilon} \zeta_{h,2}^{1}(t) +\mathcal{O} (\epsilon ),\end{aligned}
\end{equation}
and
\begin{equation*}%\label{re ne2}
\begin{aligned}&|p_{n,2}|^2
=\frac{\mathrm{sinc}^2(h\upsilon/2) }{\mathrm{sinc}^2(h\upsilon)
\epsilon^2} |\mathrm{e}^{\mathrm{i}\phi(t)/\epsilon}
\zeta_{h,2}^{1}(t) -\mathrm{e}^{-\mathrm{i}\phi(t)/\epsilon}
\zeta_{h,2}^{-1}(t)|^2 \omega^2_h(\zeta_{h,1}^{0}(t))+\mathcal{O}
(\epsilon^2 ).\end{aligned}
\end{equation*}
Furthermore, by the expansion of $q_n$ derived in Proposition
\ref{energy thm},  one gets
\begin{equation*}
\begin{aligned}|q_{n,2}|^2
=&  |\mathrm{e}^{\mathrm{i}\phi(t)/\epsilon} \zeta_{h,2}^{1}(t)
+\mathrm{e}^{-\mathrm{i}\phi(t)/\epsilon} \zeta_{h,2}^{-1}(t)|^2
+\mathcal{O}(\epsilon^3).\end{aligned}
\end{equation*}
Therefore,   inserting the above two formulae into \eqref{modi act}
yields
\begin{equation*}
\begin{aligned}I_h(q_n,p_n)%=&\frac{\Psi\mathrm{sinc} (h\upsilon)}{\mathrm{sinc}^2 (h\upsilon/2)} \frac{|p_{n,2}|^2}{
%2\omega_h(q_{n,1})} + \frac{\Psi}{\mathrm{sinc} (h\upsilon)}\frac{
%\omega_h(q_{n,1})}{2\epsilon^2}|q_{n,2}|^2\\
%=&\frac{1}{2}\mathrm{sinc}^2(h\upsilon) \cos(\frac{1}{2}h\upsilon)(
%\bar{b}_1(h\upsilon))^{-1}\frac{|p_{n,2}|^2}{\omega_h(q_{n,1})}
%+\cos(\frac{1}{2}h\upsilon)(
%\bar{b}_1(h\upsilon))^{-1}\frac{\omega_h(q_{n,1})}{2\epsilon^2}|q_{n,2}|^2\\
%=&
% \frac{1}{2\epsilon^2}\frac{1}{\mathrm{sinc} (h\upsilon)} \omega_h(\zeta_{h,1}^{0}(t))
% \Big(|\zeta_{h,2}^{1}
%-\zeta_{h,2}^{-1}|^2+|\zeta_{h,2}^{1}
%+\zeta_{h,2}^{-1}|^2\Big)+\mathcal{O}(\epsilon)\\
=& \frac{\Psi(h\upsilon,\zeta_{h,1}^{0}(t))
\mathrm{sinc}(h\upsilon/2)}{\epsilon^2} \omega_h(\zeta_{h,1}^{0}(t))
|\zeta_{h,2}^1(t)|^2
 +\mathcal{O}(\epsilon).\end{aligned}
\end{equation*}

 In what follows,
we elaborate the dominant term of $\mathcal{I}[\zeta,\eta](t)$. Fix
$t$ and consider $q^k_{h}(t)=\mathrm{e}^{\mathrm{i}k
\phi(t)/\epsilon}\zeta_h^k(t).$
 For $k \neq 0$, from Lemma 5.1 given in \cite{Hairer16}, it follows
 that
\begin{equation*}\begin{aligned}
\frac{1}{m!}\frac{d^m}{dt^m}q^k_{h}(t)=\frac{1}{m!}\zeta_h^k(t)
\big(\frac{\mathrm{i}k}{\epsilon}\dot{\phi}(t)\big)^m
\mathrm{e}^{\mathrm{i}k
\phi(t)/\epsilon}+\mathcal{O}\Big(\frac{1}{(m/M)!}\big(\frac{c}{\epsilon}\big)^{m-1-|k|}
\Big),
%\label{boun-der}%
\end{aligned}
\end{equation*}
where $c$ and the constant symbolised by $\mathcal{O}$ are
independent of $m  \geq 1$ and  $\epsilon$.

  It can be seen from inserting this into $(-1)^r
\frac{d^r}{dt^r}\big( \bar{q}^k_{h}(t)\big)^\intercal
\cos(\frac{1}{2}h\Omega)( \bar{b}_1(h\Omega))^{-1}\frac{d^s}{dt^s}
q^k_{h}(t)$  that the dominant term is to be the same whenever
$r+s=2l+ 1$. Thus,
  it is  obtained that (omit   the superscript $k$ and $(t)$ in $q^k_{h}(t)$)
\begin{equation*}
\begin{aligned}
&\Big[\big(\bar{q}_{h} \big)^\intercal \cos(\frac{1}{2}h\Omega)(
\bar{b}_1(h\Omega))^{-1}
                                                q^{(2l+1)}_{h}
                                              -\big(\bar{\dot{q}}_{h} \big)^\intercal
\cos(\frac{1}{2}h\Omega)( \bar{b}_1(h\Omega))^{-1}
                                                q^{(2l)}_{h}
 \\ &+\cdots\pm  \big( \bar{q}^{(l)}_{h} \big)^\intercal
\cos(\frac{1}{2}h\Omega)( \bar{b}_1(h\Omega))^{-1}
                                                q^{(l+1)}_{h} \Big]\\
=&(l+1)\big(\frac{\mathrm{i}k}{\epsilon}\dot{\phi} \big)^{2l+1}
 (\bar{\zeta}_h^k )^\intercal \cos(\frac{1}{2}h\Omega)(
\bar{b}_1(h\Omega))^{-1}\zeta_h^k
+\mathcal{O}\Big(\frac{1}{(l/M)!}\big(\frac{c}{\epsilon}\big)^{2l-2|k|}
\Big).
\end{aligned}
\end{equation*}
This implies  that the total derivative of \eqref{ld1} is given by
\begin{equation}\label{one main part}
\begin{aligned}
&\frac{\mathrm{i}}{2\epsilon} \sum\limits_{|k|\leq N+1}k
\Big[-\sum\limits_{l\geq0}\frac{2h^{2l}}{(2l+2)!}\Big(\big(\bar{q}_{h}(t)\big)^\intercal
\cos(\frac{1}{2}h\Omega)( \bar{b}_1(h\Omega))^{-1}
                                                q^{(2l+1)}_{h}(t)
                                               \\ &-\big(\bar{\dot{q}}_{h}(t)\big)^\intercal
\cos(\frac{1}{2}h\Omega)( \bar{b}_1(h\Omega))^{-1}
                                                q^{(2l)}_{h}(t)
+\cdots\pm  \big( \bar{q}^{(l)}_{h}(t)\big)^\intercal
\cos(\frac{1}{2}h\Omega)\\
&( \bar{b}_1(h\Omega))^{-1}
                                                q^{(l+1)}_{h}(t)\Big)\Big]+\mathcal{O}(\epsilon)\\
=&\frac{-\mathrm{i}}{2\epsilon}\sum\limits_{0<|k|\leq N+1} \Big[
\frac{\mathrm{i}k}{h} \sum\limits_{l\geq0}\frac{(-1)^{l}}{(2l+1)!}
\big(\frac{kh}{\epsilon}\dot{\phi}(t)\big)^{2l+1}
(\bar{\zeta}_h^k(t))^\intercal \cos(\frac{1}{2}h\Omega)(
\bar{b}_1(h\Omega))^{-1}\\
& \zeta_h^k(t)+\mathcal{O}(\epsilon^{2|k|}) \Big]+\mathcal{O}(\epsilon)\\
=&  \frac{1}{2\epsilon h} \sum\limits_{0<|k|\leq N+1} \Big[ k
\sin\big(\frac{kh}{\epsilon}\dot{\phi}(t)\big)
(\bar{\zeta}_h^k(t))^\intercal \cos(\frac{1}{2}h\Omega)(
\bar{b}_1(h\Omega))^{-1}\zeta_h^k(t) +\mathcal{O}(h\epsilon^{2|k|})
\Big]\\
&+\mathcal{O}(\epsilon)\\
%=& \frac{1}{\epsilon h}
%\sin\big(\frac{h}{\epsilon}\dot{\phi}(t)\big)
%(\bar{\zeta}_h^1(t))^\intercal \cos(\frac{1}{2}h\Omega)(
%\bar{b}_1(h\Omega))^{-1}\zeta_h^1(t) +\mathcal{O}(\epsilon)\\
=& \frac{1}{\epsilon h}
\sin\big(\frac{h}{\epsilon}\dot{\phi}(t)\big)
  \cos(\frac{1}{2}h\upsilon)(
\bar{b}_1(h\upsilon))^{-1}|\zeta_{h,2}^1(t)|^2
+\mathcal{O}(\epsilon).
\end{aligned}
\end{equation}
In a similar way,  we  can obtain the total derivative of
\eqref{1sta er 2} as follows
\begin{equation*}
\begin{aligned}
 &\frac{h}{\epsilon}
\sin\big(\frac{h}{\epsilon}\dot{\phi}(t)\big)
\frac{1}{2}\mathrm{sinc}(\frac{1}{2}h\upsilon)(
b_1(h\upsilon))^{-1}|\eta_{h,2}^1(t)|^2  +\mathcal{O}(\epsilon).
\end{aligned}
\end{equation*}
Combine these two total derivatives together and then we obtain
\begin{equation*}
\begin{aligned}
&\mathcal{I}[\zeta,\eta](t) =\frac{1}{\epsilon h}
\sin\big(\frac{h}{\epsilon}\dot{\phi}(t)\big)
  \cos(\frac{1}{2}h\upsilon)(
\bar{b}_1(h\upsilon))^{-1}|\zeta_{h,2}^1(t)|^2
\\
&+\frac{h}{\epsilon} \sin\big(\frac{h}{\epsilon}\dot{\phi}(t)\big)
\frac{1}{2}\mathrm{sinc}(\frac{1}{2}h\upsilon)(
b_1(h\upsilon))^{-1}|\eta_{h,2}^1(t)|^2  +\mathcal{O}(\epsilon)\\
% =&\frac{1}{\epsilon h}
%\sin\big(\frac{h}{\epsilon}\dot{\phi}(t)\big)\Big[
%  \cos(\frac{1}{2}h\upsilon)(
%\bar{b}_1(h\upsilon))^{-1}
%\\
%&+ \frac{1}{2}h^2\mathrm{sinc}(\frac{1}{2}h\upsilon)(
%b_1(h\upsilon))^{-1}\frac{\mathrm{sinc}^2(h\upsilon/2)}{\mathrm{sinc}^2(h\upsilon)}\frac{\omega_h^2(\zeta_{h,1}^{0}(t))}{
%\epsilon^2} \Big]|\zeta_{h,2}^1(t)|^2  +\mathcal{O}(\epsilon)\\
 =&\frac{1}{\epsilon h}
\sin\big(\frac{h}{\epsilon}\dot{\phi}(t)\big)\Big[
  \cos(\frac{1}{2}h\upsilon)(
\bar{b}_1(h\upsilon))^{-1}
\\
&+ \frac{1}{2}h^2\upsilon^2\mathrm{sinc}(\frac{1}{2}h\upsilon)(
b_1(h\upsilon))^{-1}\Big(\frac{\mathrm{sinc}^2(h\upsilon/2)}{\mathrm{sinc}^2(h\upsilon)}\frac{\omega_h^2(\zeta_{h,1}^{0}(t))}{
\omega_0^2}\Big) \Big]|\zeta_{h,2}^1(t)|^2  +\mathcal{O}(\epsilon),
\end{aligned}
\end{equation*}
where  the connection \eqref{re ne1} between $\eta_{h,2}^1$ and
$\zeta_{h,2}^1$    was used here.

 %It follows from
%\begin{equation*}
%\begin{aligned}
%\frac{\mathrm{sinc}^2(h\upsilon/2)}{\mathrm{sinc}^2(h\upsilon)}\frac{\omega_h^2(\zeta_{h,1}^{0}
%)}{ \omega_0^2}&=\frac{\omega^2(\zeta_{h,1}^{0} )}{
%\omega_0^2}\frac{\mathrm{sinc}^2(h\upsilon/2)}{\mathrm{sinc}^2(h\upsilon)}\Big(
%1-\frac{h^2}{4\epsilon^2} \mathrm{sinc}^2(h\upsilon/2)
%\omega^2(\zeta_{h,1}^{0})\Big)
%\end{aligned}
%\end{equation*}
%that this expression  is  of the form $1+\mathcal{O}(\epsilon)$
%based on the condition \eqref{numerical cond2} and the fact that
%$\omega(q_1(t))$ is a slowly changing frequency. Thence one arrives
%at
%\begin{equation*}
%\begin{aligned}
%\mathcal{I}[\zeta,\eta](t)
% =&\frac{1}{\epsilon h}
%\sin\big(\frac{h}{\epsilon}\dot{\phi}(t)\big)\Big[
%  \cos(\frac{1}{2}h\upsilon)(
%\bar{b}_1(h\upsilon))^{-1} +
%\frac{1}{2}h^2\upsilon^2\mathrm{sinc}(\frac{1}{2}h\upsilon)(
%b_1(h\upsilon))^{-1} \Big]\\
% &|\zeta_{h,2}^1(t)|^2
%+\mathcal{O}(\epsilon) =\frac{1}{\epsilon
%h}\frac{2\sin\big(\frac{h}{\epsilon}\dot{\phi}(t)\big)
%}{\mathrm{sinc}(h\upsilon)
%\mathrm{sinc}(h\upsilon/2)}|\zeta_{h,2}^1(t)|^2
%+\mathcal{O}(\epsilon).
%\end{aligned}
%\end{equation*}
In the light of \eqref{phase func} and the following fact
\begin{equation*}
 \sin(\frac{h\dot{\phi}(t)}{\epsilon}
 )=2\sin(\frac{h\dot{\phi}(t)}{\epsilon}
 ) \cos(\frac{h\dot{\phi}(t)}{\epsilon}
 )=\frac{h}{\epsilon}
\mathrm{sinc}(h\upsilon/2) \omega_h(\zeta_{h,1}^{0}(t)),
\end{equation*}
it is  confirmed  that
\begin{equation*}
\begin{aligned}
\mathcal{I}[\zeta,\eta](t)
=&\frac{\Psi(h\upsilon,\zeta_{h,1}^{0}(t))}{\epsilon^2}
\mathrm{sinc}(h\upsilon/2)
\omega_h(\zeta_{h,1}^{0}(t))|\zeta_{h,2}^1(t)|^2
+\mathcal{O}(\epsilon)=I_h(q_n,p_n)+\mathcal{O}(\epsilon).
\end{aligned}
\end{equation*}
The proof is complete.
 \hfill
\end{proof}

\begin{rem} Based on the above proposition and following the same approach used in Section XIII.7 of
\cite{hairer2006},  the long-time near-conservation of the modified
action $I_h$ given in Theorem \ref{pre mod I} can be obtained.
\end{rem}

\section{Energy conservation of the integrator (proof of Theorem \ref{pre mod H})} \label{sec:energy-invariants}
In this section, we prove the  near-conservation of a modified
energy for the ERKN integrator.
 We   have the following
almost-invariant of the modulated Fourier expansion.
\begin{prop}\label{H invariant thm}
Under the conditions of Theorem \ref{second invariant thm}, there
exists a function $\mathcal{H}[\zeta,\eta]$ such that
\begin{equation*}
\begin{aligned}
&\mathcal{H}[\zeta,\eta](t)=\mathcal{H}[\zeta,\eta](0)+\mathcal{O}(t\epsilon^{N}) \ \ \textmd{for}\ 0\leq t\leq h,\\
&\mathcal{H}[\zeta,\eta](nh)=H_h(q_n,p_n)+\mathcal{O}(\epsilon) \ \
\textmd{for}\ n=0,1,
\label{HH}%
\end{aligned}
\end{equation*}
where the constants  symbolised by $\mathcal{O}$ are independent of
$n,\ h$ and $\epsilon$, and  $\mathcal{H}[\zeta,\eta]$ depends on
$\zeta, \dot{\zeta},\ldots, \zeta^{(M)}$  and
 $\eta,  \dot{\eta},\ldots, \eta^{(M)}$.
\end{prop}
\begin{proof}
\textit{\textbf{Proof of the first statement.}} Similarly to that of
Theorem \ref{second invariant thm}, it is easy to get
\begin{eqnarray}
 &\ & \frac{1}{2}\frac{d}{dt}\Big[\frac{1}{\epsilon^2}
\mathcal{V}( \tilde{q}(t+\frac{h}{2}))+
 \mathcal{U}( \tilde{q}(t+\frac{h}{2}))+\frac{1}{\epsilon^2} \mathcal{V}(
 \tilde{q}(t-\frac{h}{2}))+
 \mathcal{U}(
\tilde{q}(t-\frac{h}{2}))\Big]  \nonumber \\
   &= & \frac{1}{2}\sum\limits_{|k|\leq N+1}
\big(\dot{q}^{-k}_{h}(t)\big)^\intercal
\cos(\frac{1}{2}h\Omega)(-h^2\bar{b}_1(h\Omega))^{-1}
\mathcal{L}(hD) q^k_{h}(t)  \label{sta er 1}\\
    &\  & +\frac{1}{2}\sum\limits_{|k|\leq N+1}
\big(\dot{p}^{-k}_h(t)\big)^\intercal \frac{1}{2}h
\mathrm{sinc}(\frac{1}{2}h\Omega)(-h
b_1(h\Omega))^{-1}\mathcal{L}(hD) p^k_{h}(t) \label{sta er 2}\\
    &\ & +\mathcal{O}(\epsilon^{N}). \nonumber
\end{eqnarray}
Using again the ``magic formulas" on p. 508 of \cite{hairer2006}, we
deduce that (omitting the superscript $k$ and $t$ on $q^{k}(t)$)
\begin{equation}\label{en de1}
\begin{aligned}
& \textmd{the term}\  \eqref{sta er 1}\\
%=&\frac{1}{2}\sum\limits_{|k|\leq
%N+1}\mathrm{Re}\Big[\big(\bar{\dot{q}}_{h} \big)^\intercal
%\cos(\frac{1}{2}h\Omega)(-h^2\bar{b}_1(h\Omega))^{-1}
% \Big(4\sin^2(\frac{1}{2}h\Omega )q_{h} \\
% &+\sum\limits_{l\geq0}\frac{2h^{2l+2}}{(2l+2)!}
%                                                q^{(2l+2)}_{h} \Big)\Big],\\
=&\frac{1}{2}\sum\limits_{|k|\leq
N+1}\mathrm{Re}\Big[\frac{1}{2}\frac{d}{dt}\Big(\big(\bar{q}_{h}(t)
\big)^\intercal
\cos(\frac{1}{2}h\Omega)(-h^2\bar{b}_1(h\Omega))^{-1}
  4\sin^2(\frac{1}{2}h\Omega )q_{h} \Big)\Big]\\
  &-\frac{1}{2}\sum\limits_{|k|\leq N+1}\frac{d}{dt}\sum\limits_{l\geq0}\frac{2h^{2l}}{(2l+2)!}\mathrm{Re}\Big[\big(\bar{\dot{q}}_{h} \big)^\intercal
\cos(\frac{1}{2}h\Omega)( \bar{b}_1(h\Omega))^{-1}
                                                q^{(2l+1)}_{h}
                                               \\ &-\big(\bar{\ddot{q}}_{h} \big)^\intercal
\cos(\frac{1}{2}h\Omega)( \bar{b}_1(h\Omega))^{-1}
                                                q^{(2l)}_{h}
+\cdots\mp  \big( \bar{q}^{(l)}_{h} \big)^\intercal
\cos(\frac{1}{2}h\Omega)( \bar{b}_1(h\Omega))^{-1}
                                                q^{(l+2)}_{h} \\
 &\pm \frac{1}{2} \big( \bar{q}^{(l+1)}_{h} \big)^\intercal
\cos(\frac{1}{2}h\Omega)( \bar{b}_1(h\Omega))^{-1}
                                                q^{(l+1)}_{h} \Big],
\end{aligned}
\end{equation}
which is a total derivative. In a similar way,  we obtain that the
term  of   \eqref{sta er 2}   is also a total derivative.
 Therefore,  there is a smooth function $\mathcal{K}[\zeta,\eta](t)$ such that
\begin{equation*}
\begin{aligned}\frac{d}{dt}\mathcal{K}[\zeta,\eta](t)=&-\frac{1}{2}\frac{d}{dt}\Big[\frac{1}{\epsilon^2}
\mathcal{V}( \zeta(t+\frac{h}{2}))+
 \mathcal{U}( \zeta(t+\frac{h}{2}))\\
&+\frac{1}{\epsilon^2} \mathcal{V}(
 \zeta(t-\frac{h}{2}))+
 \mathcal{U}(
\zeta(t-\frac{h}{2}))\Big]+\mathcal{O}(\epsilon^{N})\end{aligned}
\end{equation*} and then
$\frac{d}{dt}\mathcal{H}[\zeta,\eta](t)=\mathcal{O}(\epsilon^{N}).$
An integration yields the first statement  of the theorem.

\textit{\textbf{Proof of the second statement.}}
%Omitting the
%superscript $k$ and $(t)$ in $q_{h}(t)$, we have
%\begin{equation*}
%\begin{aligned}
%&\Big[\big(\bar{\dot{q}}_{h} \big)^\intercal
%\cos(\frac{1}{2}h\Omega)( \bar{b}_1(h\Omega))^{-1}
%                                                q^{(2l+1)}_{h}
%                                            -\big( \bar{\ddot{q}} _{h} \big)^\intercal
%\cos(\frac{1}{2}h\Omega)( \bar{b}_1(h\Omega))^{-1}
%                                                q^{(2l)}_{h}
%  +\cdots\\
%  &\mp  \big( \bar{q}^{(l)}_{h} \big)^\intercal
%\cos(\frac{1}{2}h\Omega)( \bar{b}_1(h\Omega))^{-1}
%                                                q^{(l+2)}_{h} \pm \frac{1}{2} \big( \bar{q}^{(l+1)}_{h} \big)^\intercal
%\cos(\frac{1}{2}h\Omega)( \bar{b}_1(h\Omega))^{-1}
%                                                q^{(l+1)}_{h} \Big]\\
%=&-(l+\frac{1}{2})\big(\frac{\mathrm{i}k}{\epsilon}\dot{\phi}
%\big)^{2l+2}
% (\bar{\zeta}_h^k )^\intercal \cos(\frac{1}{2}h\Omega)(
%\bar{b}_1(h\Omega))^{-1}\zeta_h^k
%+\mathcal{O}\Big(\frac{1}{(l/M)!}\big(\frac{c}{\epsilon}\big)^{2l+1-2|k|}
%\Big).
%\end{aligned}
%\end{equation*}
Similarly to that of Theorem \ref{second invariant thm}, the total
derivative of \eqref{en de1} is
\begin{equation*}
\begin{aligned}
&\frac{1}{h^2} \big(\bar{q}^0_{h}(t)\big)^\intercal
\cos(\frac{1}{2}h\Omega)(\bar{b}_1(h\Omega))^{-1}
 \sin^2(\frac{1}{2}h\Omega)q^0_{h}(t)\\
&+ \frac{1}{4} \big(\bar{q}^1_{h}(t)\big)^\intercal
\cos(\frac{1}{2}h\Omega)(\bar{b}_1(h\Omega))^{-1}
 \sin^2(\frac{1}{2}h\Omega)q^1_{h}(t)\\
&+\sum\limits_{0<|k|\leq N+1}\Big[\frac{1}{h^2}
\big(\bar{q}^k_{h}(t)\big)^\intercal
\cos(\frac{1}{2}h\Omega)(\bar{b}_1(h\Omega))^{-1}
 \sin^2(\frac{1}{2}h\Omega)q^k_{h}(t)\\
  &- \frac{1}{2}
 \sum\limits_{l\geq0}\frac{2h^{2l}}{(2l+2)!}(l+\frac{1}{2})
\big(\frac{\mathrm{i}k}{\epsilon}\dot{\phi}(t)\big)^{2l+2}
(\bar{\zeta}_h^k(t))^\intercal \cos(\frac{1}{2}h\Omega)(
\bar{b}_1(h\Omega))^{-1}\zeta_h^k(t)\Big]+\mathcal{O}(\epsilon).
\end{aligned}
\end{equation*}
In the light of the bounds  \eqref{coefficient func1} and
\eqref{coefficient func2},  the dominant terms of above expressions
are
\begin{equation*}
\begin{aligned}
&\frac{1}{2} |\dot{\zeta}_{h,1}^0(t)|^2+ \frac{1}{h^2}
\cos(\frac{1}{2}h\upsilon)(\bar{b}_1(h\upsilon))^{-1}
 \sin^2(\frac{1}{2}h\upsilon) |\zeta_{h,2}^1(t)|^2\\
  &- \frac{1}{2}
 \sum\limits_{l\geq0}\frac{2h^{2l}}{(2l+2)!}(l+\frac{1}{2})
\big(\frac{\mathrm{i}k}{\epsilon}\dot{\phi}(t)\big)^{2l+2}
  \cos(\frac{1}{2}h\upsilon)(
\bar{b}_1(h\upsilon))^{-1}|\zeta_{h,2}^1(t)|^2
+\mathcal{O}(\epsilon).
\end{aligned}
\end{equation*}
We split $(l+\frac{1}{2})=(l+1)-\frac{1}{2}$.  Then the above terms
become
\begin{equation*}
\begin{aligned}
&\frac{1}{2} |\dot{\zeta}_{h,1}^0(t)|^2+  \frac{1}{h^2}
\cos(\frac{1}{2}h\upsilon)(\bar{b}_1(h\upsilon))^{-1}
 \sin^2(\frac{1}{2}h\upsilon) |\zeta_{h,2}^1(t)|^2\\
 &+    \frac{  \dot{\phi}(t)}{2h\epsilon}
\sum\limits_{l\geq0}\frac{(-1)^{l}}{(2l+1)!}  \big(\frac{
h}{\epsilon}\dot{\phi}(t)\big)^{2l+1} \cos(\frac{1}{2}h\upsilon)(
\bar{b}_1(h\upsilon))^{-1}|\zeta_{h,2}^1(t)|^2\\
&+    \frac{1}{2h^2}
 \sum\limits_{l\geq0}\frac{(-1)^{l+1}}{(2l+2)!} \big(\frac{ k
h}{\epsilon}\dot{\phi}(t)\big)^{2l+2}  \cos(\frac{1}{2}h\upsilon)(
\bar{b}_1(h\upsilon))^{-1}|\zeta_{h,2}^1(t)|^2 +\mathcal{O}(\epsilon)\\
%=&\frac{1}{2} |\dot{\zeta}_{h,1}^0(t)|^2+\frac{2}{h^2}
%\sin^2(\frac{1}{2}h\upsilon)\cos(\frac{1}{2}h\upsilon)(\bar{b}_1(h\upsilon))^{-1}
%|\zeta_{h,2}^1(t)|^2\\&+ \frac{1}{h\epsilon}\dot{\phi}(t) \sin
%\big(\frac{ h}{\epsilon}\dot{\phi}(t)\big)
%\cos(\frac{1}{2}h\upsilon)(
%\bar{b}_1(h\upsilon))^{-1} |\zeta_{h,2}^1(t)|^2\\
%&- \frac{1}{h^2} \big(1-\cos(\frac{ h}{\epsilon}\dot{\phi}(t))\big)
%\cos(\frac{1}{2}h\upsilon)( \bar{b}_1(h\upsilon))^{-1}
%|\zeta_{h,2}^1(t)|^2  +\mathcal{O}(\epsilon)\\
=&\frac{1}{2} |\dot{\zeta}_{h,1}^0(t)|^2 +
\frac{1}{h\epsilon}\dot{\phi}(t) \sin \big(\frac{
h}{\epsilon}\dot{\phi}(t)\big) \cos(\frac{1}{2}h\upsilon)(
\bar{b}_1(h\upsilon))^{-1} |\zeta_{h,2}^1(t)|^2\\
&- \frac{1}{h^2} \big(1-2\sin^2(\frac{1}{2}h\upsilon)-\cos(\frac{
h}{\epsilon}\dot{\phi}(t))\big) \cos(\frac{1}{2}h\upsilon)(
\bar{b}_1(h\upsilon))^{-1} |\zeta_{h,2}^1(t)|^2
 +\mathcal{O}(\epsilon).
\end{aligned}
\end{equation*}
In a similar way, the total derivative of \eqref{sta er 2}   can be
obtained
  as follows:
\begin{equation*}
\begin{aligned}
%& \frac{1}{h\epsilon}\dot{\phi}(t) \sin \big(\frac{
%h}{\epsilon}\dot{\phi}(t)\big) \frac{1}{2}h^2
%\mathrm{sinc}(\frac{1}{2}h\upsilon)( b_1(h\upsilon))^{-1}
%|\eta_{h,2}^1(t)|^2- \frac{1}{h^2}
%\big(1-\\
%&2\sin^2(\frac{1}{2}h\upsilon)-\cos(\frac{
%h}{\epsilon}\dot{\phi}(t))\big)  \frac{1}{2}h^2
%\mathrm{sinc}(\frac{1}{2}h\upsilon)( b_1(h\upsilon))^{-1}
%|\eta_{h,2}^1(t)|^2
% +\mathcal{O}(\epsilon)\\
%=
& \frac{1}{h\epsilon}\dot{\phi}(t) \sin \big(\frac{
h}{\epsilon}\dot{\phi}(t)\big) \frac{1}{2}h^2\upsilon^2
\mathrm{sinc}(\frac{1}{2}h\upsilon)( b_1(h\upsilon))^{-1}
|\zeta_{h,2}^1(t)|^2- \frac{1}{h^2}
\big(1-2\sin^2(\frac{1}{2}h\upsilon)\\
&-\cos(\frac{ h}{\epsilon}\dot{\phi}(t))\big)
\frac{1}{2}h^2\upsilon^2 \mathrm{sinc}(\frac{1}{2}h\upsilon)(
b_1(h\upsilon))^{-1} |\zeta_{h,2}^1(t)|^2
 +\mathcal{O}(\epsilon).
\end{aligned}
\end{equation*}
It follows from \eqref{fir phi} that
\begin{equation*}
\begin{aligned}
1-2\sin^2(\frac{1}{2}h\upsilon)-\cos(\frac{
h}{\epsilon}\dot{\phi}(t))&= \cos(\frac{1}{2}h\upsilon)-\cos(\frac{
h}{\epsilon}\dot{\phi}(t))\\
&=\kappa^2\bar{b}_1(h\upsilon)\big(\omega^2(\xi_{h,1}^{0}(t))-\omega_0^2\big).
\end{aligned} %
\end{equation*}
We then have
\begin{equation*}
\begin{aligned}
 \mathcal{K}[\zeta,\eta](t)  =&\frac{1}{2}
|\dot{\zeta}_{h,1}^0(t)|^2 + \frac{1}{h\epsilon}\dot{\phi}(t) \sin
\big(\frac{ h}{\epsilon}\dot{\phi}(t)\big)\Big(
\cos(\frac{1}{2}h\upsilon)(
\bar{b}_1(h\upsilon))^{-1}\\
&+\frac{1}{2}h^2\upsilon^2 \mathrm{sinc}(\frac{1}{2}h\upsilon)(
b_1(h\upsilon))^{-1}\Big) |\zeta_{h,2}^1(t)|^2
\\
&- \frac{1}{h^2} \big(1-2\sin^2(\frac{1}{2}h\upsilon)-\cos(\frac{
h}{\epsilon}\dot{\phi}(t))\big) \Big( \cos(\frac{1}{2}h\upsilon)(
\bar{b}_1(h\upsilon))^{-1}\\
&+\frac{1}{2}h^2\upsilon^2 \mathrm{sinc}(\frac{1}{2}h\upsilon)(
b_1(h\upsilon))^{-1}\Big) |\zeta_{h,2}^1(t)|^2.
%=&\frac{1}{2} |\dot{\zeta}_{h,1}^0(t)|^2 +
%\frac{1}{h\epsilon}\dot{\phi}(t) \sin \big(\frac{
%h}{\epsilon}\dot{\phi}(t)\big)
%\frac{2}{\mathrm{sinc}(h\upsilon)\mathrm{sinc}(\frac{1}{2}h\upsilon)}
%|\zeta_{h,2}^1(t)|^2
%\\& + \frac{1}{\epsilon^2} \frac{\mathrm{sinc}(\frac{1}{2}h\upsilon)}{\mathrm{sinc}(h\upsilon)}
%\big(\omega_0^2-\omega^2(\xi_{h,1}^{0}(t))\big) |\zeta_{h,2}^1(t)|^2
%+\mathcal{O}(\epsilon).
\end{aligned}
\end{equation*}
Therefore, one arrives that
\begin{equation*}
\begin{aligned}
&\mathcal{H}[\zeta,\eta](t)\\
%=&\frac{1}{2} |\dot{\zeta}_{h,1}^0(t)|^2 +
%\frac{1}{h\epsilon}\dot{\phi}(t) \sin \big(\frac{
%h}{\epsilon}\dot{\phi}(t)\big) \cos(\frac{1}{2}h\upsilon)(
%\bar{b}_1(h\upsilon))^{-1}  |\zeta_{h,2}^1(t)|^2\\
%&+ \frac{1}{\epsilon^2}
%  \cos(\frac{1}{2}h\upsilon)
%\big(\omega_0^2-\omega^2(\xi_{h,1}^{0}(t))\big)
% |\zeta_{h,2}^1(t)|^2\\
%&+\frac{1}{2} \Big[\frac{1}{\epsilon^2} \mathcal{V}(
%\zeta(t+\frac{h}{2}))+
% \mathcal{U}(\zeta(t+\frac{h}{2}))+\frac{1}{\epsilon^2} \mathcal{V}(
%\zeta(t-\frac{h}{2}))+
% \mathcal{U}(
%\zeta(t-\frac{h}{2}))\Big]+\mathcal{O}(\epsilon)\\
=&\frac{1}{2}|\dot{\zeta}_{h,1}^0(t)|^2+
\tilde{\omega}_h(\zeta_{h,1}^{0}(t))\mathcal{I}[\zeta,\eta](t)+
\frac{1}{\epsilon^2}\Psi(h\upsilon,\zeta_{h,1}^{0}(t))\bar{b}_1(h\upsilon)
\big(\omega_0^2-\omega^2(\zeta_{h,1}^{0}(t))\big)
  |\zeta_{h,2}^1(t)|^2
\\
&+  \Big[\frac{1}{\epsilon^2} \mathcal{V}( \zeta(t))+
 \mathcal{U}(\zeta(t))\Big]+\mathcal{O}(\epsilon)\\
 =&\frac{1}{2}|\dot{\zeta}_{h,1}^0(t)|^2+   \tilde{\omega}_h(\zeta_{h,1}^{0}(t)) \mathcal{I}[\zeta,\eta](t)
 +  \Psi(h\upsilon,\zeta_{h,1}^{0}(t)) \bar{b}_1(h\upsilon)
\frac{\omega_0^2-\omega^2(\zeta_{h,1}^{0}(t))}{\epsilon^2}
  |\zeta_{h,2}^1(t)|^2
\\
&+ \Big[\frac{\omega^2(\zeta_{h,1}^{0}(t))-\omega_0^2}{\epsilon^2}
|\zeta_{h,2}^1(t)|^2+
 U(\zeta(t))\Big]+\mathcal{O}(\epsilon)\\
  =&\frac{1}{2}|\dot{\zeta}_{h,1}^0(t)|^2
 +   \tilde{\omega}_h(\zeta_{h,1}^{0}(t))
\mathcal{I}[\zeta,\eta](t) +
U(\zeta(t))\\
 &+ \big( 1- \Psi(h\upsilon,\zeta_{h,1}^{0}(t)) \bar{b}_1(h\upsilon)\big)\frac{\omega^2(\zeta_{h,1}^{0}(t))-\omega_0^2}{\epsilon^2}
|\zeta_{h,2}^1(t)|^2 +\mathcal{O}(\epsilon).
\end{aligned}
\end{equation*}
   The following result is used in the above derivation
$$\mathcal{V}(
\zeta(t))=\frac{\omega^2(\zeta_{h,1}^{0}(t))-\omega_0^2}{\epsilon^2}
|\zeta_{h,2}^1(t)|^2+\mathcal{O}(\epsilon),$$ which is obtained by
considering the definition of $\mathcal{V}$ and the bounds given in
Theorem \ref{second invariant thm}.

In terms of  the above analysis and Theorem \ref{second invariant
thm}, it is easy to obtain the second statement of the theorem.
 \hfill
 \end{proof}

 \begin{rem}
Following Section XIII. 7 of \cite{hairer2006} again, we can obtain
a long-time near-conservation of the modified energy $H_h$ as stated
in Theorem \ref{pre mod H}.
 \end{rem}

\section{Conclusions} \label{sec:conclusions}

The long-time behaviour    of an ERKN integrator was analysed
rigorously for solving highly oscillatory Hamiltonian systems with a
slowly varying, solution-dependent high frequency. A
varying-frequency modulated Fourier expansion of the ERKN integrator
was developed and  studied. On the basis of the expansion, we showed
that the symmetric ERKN integrator has two almost-invariants and
approximately conserves the modified action and energy over long
times.

 This paper also shown an important observation from the numerical
experiment that the numerical behaviour of the ERKN integrator is
better than that of the SV method when solving highly oscillatory
Hamiltonian systems with a slowly varying, solution-dependent high
frequency (see Fig. \ref{P5-2}, Fig. \ref{fig0}, Tables \ref{ta
re1}-\ref{ta re2}).

Last but not least, we note that   as a preliminary  research on the
long-term analysis of  ERKN integrators for  Hamiltonian systems
with a solution-dependent high frequency, this paper is restricted
to the exemplary case of the ERKN integrators with a fixed
frequency. It is interesting but not so easy to extend the analysis
of the present paper to ERKN integrators with a fixed but different
frequency at each step. This issue will be our next investigation in
future. If the high frequency also depends on $p$, the Hamiltonian
ordinary differential equations  will become a new scheme (not a
second-order system any more). Both the integrators and the analysis
presented in this paper are not applicable. The investigation of
this general system will be our another issue considered in future.

\section*{Acknowledgements}

The authors are sincerely thankful to  Professor Christian Lubich
for his helpful comments and discussions on the topic of modulated
Fourier expansions. We also thank him for drawing our attention to
Hamiltonian systems with a solution-dependent high frequency.

\end{document}